\documentclass[11pt,reqno,a4paper]{amsart}

\usepackage{fbb}
\usepackage[utf8]{inputenc}
\usepackage[T1]{fontenc}
\usepackage{amsmath,amsthm,amssymb,amsfonts,mathtools,mathrsfs}
\usepackage{microtype}
\usepackage{graphicx}
\usepackage{hyperref}
\hypersetup{
  colorlinks=true,
  linkcolor=blue!50!green,
  citecolor=blue!50!green,
  urlcolor=blue!50!green,
  pdfauthor={Mayukh Mukherjee},
  pdftitle={An Approximate Inverse Spectral Theorem for Manifolds of Constant Negative Curvature}
}
\usepackage[capitalise,nameinlink]{cleveref}

\usepackage[most]{tcolorbox}

\usepackage[top=2.5cm,bottom=2.4cm,left=2.6cm,right=2.6cm,headsep=0.2in]{geometry}
\newtheorem{theorem}{Theorem}[section]
\newtheorem{lemma}[theorem]{Lemma}
\newtheorem{corollary}[theorem]{Corollary}
\newtheorem{proposition}[theorem]{Proposition}
\theoremstyle{definition}

\newtheorem{remark}[theorem]{Remark}

\numberwithin{equation}{section}

\newcommand{\Vol}{\operatorname{Vol}}
\newcommand{\diam}{\operatorname{diam}}
\newcommand{\supp}{\operatorname{supp}}
\usepackage{fancyhdr}
\title[Approximate Inverse Spectral Problem]{An Approximate Inverse Spectral Theorem for Manifolds of Constant Negative Curvature}
\author{Mayukh Mukherjee}

\begin{document}

\begin{abstract}
A classical theorem of Colin de Verdi\`ere shows that on a closed manifold of fixed topology one can prescribe an arbitrary finite portion of the Laplace-Beltrami spectrum (including multiplicities, subject to the usual topological constraints) by choosing a sufficiently heterogeneous smooth metric. In this paper, we study the same inverse problem under the rigid geometric constraint of \emph{constant negative sectional curvature}. Allowing the topological complexity to vary, we prove that any finite strictly increasing target list can be approximated to arbitrary precision by a closed manifold of constant negative curvature in any dimension $d\ge2$. In $d=2$ the construction uses hyperbolic collar degeneration and the discrete spectral limit theorems of Burger, building on the collar estimates of Buser; in $d\ge3$ we build macroscopically heterogeneous hyperbolic covering manifolds assembled from ``heavy'' vertex clusters and ``long'' corridor chains whose low-energy limit is a prescribed \emph{discrete} graph Laplacian. We also record the universal obstructions at curvature normalization $\kappa\equiv -1$: Yang-Yau in $d=2$ and Kazhdan-Margulis combined with Bishop--Gromov volume comparison in $d\ge3$. In particular, $\lambda_1$ is universally bounded at $\kappa=-1$, so target lists whose first positive eigenvalue exceeds this bound cannot be approximated within the class $\kappa\equiv -1$, and accommodating arbitrarily large prescribed $\lambda_1^*$ forces $|\kappa|\to\infty$. A corollary on the arbitrarily precise prescription of scale-invariant eigenvalue ratios at $\kappa\equiv -1$ and an explicit worked example are included.
\end{abstract}

\maketitle
\section{Introduction}

The inverse spectral problem for the Laplace-Beltrami operator asks to what
extent the geometry and topology of a Riemannian manifold can be prescribed to
produce a specific target spectrum. In a seminal paper \cite{CdV1987}, Colin de
Verdi\`ere proved that on a compact connected manifold of fixed topology one can
choose a smooth metric whose first finitely many non-zero Laplace eigenvalues
realize any prescribed finite list, with prescribed multiplicities subject to
the usual topological constraints. In this paper we restrict to the
simple-spectrum case; in particular, on a closed surface of genus $\gamma\ge2$
one may realize any strictly increasing list $0<\lambda_1^*<\dots<\lambda_n^*$
as the first $n$ non-zero eigenvalues.

The metrics produced by this general construction are highly heterogeneous. A
natural geometric question is whether a similar spectral prescription remains
possible under the rigid constraint of \emph{constant negative
sectional curvature}. This is a genuinely restrictive condition: for each fixed topological type,
constant-curvature metrics form a finite-dimensional deformation space
(parametrized by the curvature scale and, in $d=2$, by the Teichm\"uller
moduli), in contrast to the infinite-dimensional freedom of general
Riemannian metrics.

In dimension $d=2$, constant curvature $-1$ surfaces still admit continuous
deformations (Teichm\"uller space has dimension $6\gamma-6$ for $\gamma\ge 2$), and classical
collar degeneration provides a mechanism for spectral control: as a separating
geodesic is pinched, the surface develops a bottleneck whose conductance tends
to zero, producing a small eigenvalue governed by a discrete graph Laplacian.
This mechanism was analyzed by Burger~\cite{Burger1988,Burger1990}, building on the collar estimates of
Buser~\cite{Buser1992}.

In dimensions $d\ge 3$, the situation is fundamentally different because of
\textbf{Mostow Rigidity} \cite{Mostow1973}: a closed hyperbolic manifold of
dimension $d\ge 3$ admits a \emph{unique} metric of constant curvature $-1$, up
to isometry. The only remaining geometric freedom is the global curvature
scale~$\kappa$, which multiplies all eigenvalues by
$|\kappa|$. Thus, after normalizing to $\kappa\equiv -1$, there is no
continuous deformation of the spectrum at fixed topology; equivalently,
\emph{the scale-invariant eigenvalue ratios $\lambda_k/\lambda_1$ are
completely determined by the topology}. Varying the topology is therefore a
logical necessity, not merely a convenience. This makes the
higher-dimensional inverse problem qualitatively harder and motivates the
covering-space construction that is one of the main contributions of this paper.

Our main result shows that, despite the aforementioned rigidity constraints,
\emph{approximate} spectral universality holds in all dimensions $d\ge 2$,
provided the topology is allowed to vary.

\begin{theorem}[Approximate Inverse Spectral Theorem]
\label{thm:main}
Let $d \ge 2$ be an integer, and let $0 = \lambda^*_0 < \lambda^*_1 < \dotsb < \lambda^*_n$ be a finite, strictly increasing sequence of real numbers. For any error tolerance $\varepsilon > 0$, there exists a closed, connected $d$-dimensional Riemannian manifold $(M, g)$ such that:
\begin{enumerate}
    \item \emph{Geometric rigidity:} $g$ has constant negative sectional curvature $\kappa<0$.
    \item \emph{Spectral approximation:} the first $n+1$ eigenvalues of $\Delta_g$ in the standard non-decreasing enumeration (with repetitions for multiplicities) satisfy
    \[
        \lvert \lambda_i(M, g) - \lambda^*_i \rvert < \varepsilon \quad \text{for all } i = 0,1,\dots,n.
    \]
\end{enumerate}
Moreover, there exists a constant $\Lambda_d>0$ (see Section~\ref{sec:obstruction}) such that every closed $d$-manifold of constant sectional curvature $\kappa<0$ (orientable, if $d=2$) satisfies
\[
\lambda_1(M,g)\le \Lambda_d |\kappa|.
\]
In particular, targets with $\lambda_1^*>\Lambda_d$ cannot be approximated to arbitrary precision within the curvature normalization $\kappa\equiv -1$, and any approximating sequence with $\varepsilon\to 0$ must satisfy $\liminf |\kappa|\ge \lambda_1^*/\Lambda_d$.
\end{theorem}

\begin{remark}[Topology of the approximating manifolds]
\label{rem:topology_cost}
The manifolds produced by the proof may be chosen so that, in the surface case ($d=2$), the genus depends only on $n$ (and diverges only as $n\to\infty$), while in dimensions $d\ge 3$, the covering degree (and hence the unrescaled volume at curvature $-1$) grows as $\varepsilon \to 0$.
\end{remark}

For $\varepsilon$ sufficiently small the first $n+1$ eigenvalues are simple: the target list is strictly increasing, and the construction ensures $\lambda_{n+1}(M,g)$ is bounded away from $\lambda_n^*$ (see the proof of Theorem~\ref{thm:main} in Section~\ref{sec:proof}).

Since $n$ is fixed, all convergence statements are uniform over the finite spectral window $0\le i\le n$.

\emph{Notation.} Throughout the paper, $\lambda_k(M,g)$ denotes the $k$th eigenvalue of the Laplace-Beltrami operator on $(M,g)$. In Section~\ref{sec:higher} we write $\nu_k(M_m)$ for the eigenvalues of the intermediate covering manifold $(M_m,g_{M_m})$ at curvature $-1$, before the final rescaling $g_m=m^{-4}g_{M_m}$ that produces $\lambda_k(M_m,g_m)=m^4\nu_k(M_m)$.

\begin{remark}[Geometric collapse]
\label{rem:collapse}
In the higher-dimensional construction ($d\ge 3$), the approximating manifolds $(M_m,g_m)$
have curvature $\kappa_m=-m^4\to -\infty$, and their volume satisfies
$\Vol(M_m,g_m)=O(m^{3-2d})\to 0$. Indeed, the total number of
blocks in $M_m$ is $Nm^3+O(m)$ (clusters plus corridors), each of volume $V_F$
in the $\kappa=-1$ metric. After rescaling $g_m=m^{-4}g_{M_m}$, volumes scale
by $m^{-2d}$, giving $\Vol(M_m,g_m)\sim N V_F m^{3-2d}$. Since $3-2d<0$ for
$d\ge 3$, the volume tends to zero. (Volume decay alone does not force Gromov-Hausdorff collapse to a point; the diameter estimate below is the decisive input.)

The diameter of $M_m$ at curvature $-1$ is $O(m)$ (the heavy clusters have combinatorial diameter $O(\log m)$ by the expander property, and the corridors have length $O(m)$). After rescaling by $m^{-4}$, distances scale by $m^{-2}$, giving $\diam(M_m,g_m)=O(m^{-1})\to 0$. Since $\diam\to 0$, the approximating sequence collapses in the Gromov-Hausdorff sense to a single point. (Whether such collapse is unavoidable for \emph{every} approximating family with $|\kappa|\to\infty$ and bounded low eigenvalues is a natural open question; we have not investigated this.)
\end{remark}

The obstruction statement in Theorem~\ref{thm:main} concerns the
\emph{absolute} scale of the spectrum. By contrast, scale-invariant ratios such
as $\lambda_k/\lambda_1$ are unaffected by curvature rescaling. In view of
Mostow Rigidity, for a fixed topology admitting a hyperbolic metric in $d\ge 3$ these ratios are completely determined; yet we
show they can be prescribed arbitrarily by varying the topology:

\begin{corollary}[Corollary~\ref{cor:eigenvalue_ratios}, stated informally]
For any $d\ge 2$, any target ratios $1=\mu_1^*<\cdots<\mu_n^*$, and any
$\varepsilon>0$, there exists a closed $d$-manifold with $\kappa\equiv -1$
whose eigenvalue ratios $\lambda_i/\lambda_1$ approximate $\mu_i^*$ to within
$\varepsilon$.
\end{corollary}

\noindent In $d\ge 3$, this is arguably the main geometric content of the
paper: Mostow Rigidity forbids \emph{any} continuous variation of these ratios
on a fixed topology, yet by varying the topology we achieve arbitrary
prescription.

\subsection*{Strategy and relation to prior work}

The proof strategy in all dimensions is to reduce the continuous inverse
spectral problem to a \emph{discrete} one: first prescribe the spectrum of a
weighted graph Laplacian on the complete graph $K_N$ (which is possible by
Colin de Verdi\`ere's discrete inverse theorem \cite{CdV1988}), then realize
this graph Laplacian as the low-energy limit of a sequence of hyperbolic
manifolds.

In $d=2$, the realization step uses hyperbolic collar degeneration: pinching
a multicurve of $\binom{N}{2}$ disjoint geodesics whose complement has $N$
connected components and dual graph $K_N$, on a surface of genus
$\gamma=1+\frac{N(N-3)}{2}$ (each vertex piece has genus $0$ with $N-1$ boundary curves, giving Euler characteristic $3-N$; the $\binom{N}{2}$ annular collars contribute $\chi=0$, and the total $\chi=N(3-N)+0=2-2\gamma$ yields the formula), produces a discrete graph limit, following
Burger~\cite{Burger1988,Burger1990} and Buser~\cite[\S 8]{Buser1992}. Combined with
a global metric rescaling $g_\delta:=\delta g_{\Sigma_\delta}$ (which scales
the curvature to $|\kappa|\to\infty$ and the eigenvalues to their target
values), this yields Theorem~\ref{thm:main} for surfaces. The surface case is
thus a synthesis of classical results, and we present it in
Section~\ref{sec:surfaces} as a warm-up and template for the
higher-dimensional argument.

The guiding heuristic for the higher-dimensional case is this: since Mostow Rigidity forbids any deformation of the geometry, one must instead encode the spectral degrees of freedom in the way the topology is assembled: specifically, in the combinatorial pattern of a covering space. Suppose a hyperbolic manifold is built to look, at large scales, like a finite graph: heavy lumps (the vertices) connected by long thin passages (the edges). If the lumps are much heavier than the passages, then low-frequency eigenfunctions will not resolve the internal structure of any lump; they will only feel the pattern of connections, which is precisely the graph Laplacian of the underlying graph. The plan, therefore, is to engineer a hyperbolic covering manifold whose large-scale topology mirrors a target weighted graph, so that the low spectrum of the manifold approximates the spectrum of that graph.

The principal novelty of this paper lies in the case $d\ge 3$. Here collar degeneration is unavailable: Mostow Rigidity prevents any
local geometric deformation. Instead, we encode the degrees of freedom in
\emph{covering combinatorics}. Starting from a closed arithmetic hyperbolic
manifold with positive first Betti number (Millson~\cite{Millson1976}) and
passing to a finite cover to obtain a base manifold $B$ whose fundamental
group surjects onto a free group (Lubotzky~\cite{Lubotzky1996}), we build
finite covers
$M_m\to B$ governed by carefully designed Schreier graphs. Each macroscopic
vertex of $K_N$ is replaced by a ``heavy'' cluster of $m^3$ copies of a fixed
fundamental block $F$, internally wired as an expander graph; each macroscopic
edge becomes a ``long'' corridor chain of $O(m)$ blocks. The volume disparity
$m^3\gg m$, combined with the uniform Poincare gap provided by the expander wiring inside each cluster, forces low-energy eigenfunctions to be approximately constant on
clusters, with transitions governed by corridor conductances that converge to
the prescribed discrete edge weights. At a conceptual level, this is a \emph{discrete-to-continuum
homogenization} on a manifold glued from periodic cells,
where the macroscopic effective operator is the graph Laplacian $L_m$. The full spectral reduction (Theorem~\ref{thm:reduction}) is proved via min-max comparison using
the cluster Poincare inequality, quantitative trace estimates at ports, and
corridor energy bounds.

This covering construction is related to, but distinct from, several earlier
lines of work:
\begin{itemize}
\item Brooks \cite{Brooks1986} studied spectral convergence for towers of
coverings, showing that $\lambda_1$ of finite covers converges to the bottom of
the $L^2$ spectrum of the universal cover. Our construction uses coverings for a
different purpose: not to approximate a limit, but to \emph{engineer} a
prescribed discrete Laplacian by controlling the combinatorial structure of the
Schreier graph.
\item Colbois--Courtois \cite{ColboisCourtois1991} gave necessary and sufficient conditions for spectral convergence when a sequence of compact Riemannian manifolds degenerates, in the pointed Lipschitz sense, to a non-compact finite-volume limit (with applications to hyperbolic surfaces and $3$-manifolds). Their ``fleeing parts'' criterion (a Dirichlet eigenvalue lower bound on the escaping regions) is conceptually related to our corridor energy bounds, but the settings differ: Colbois--Courtois study a degeneration limit, whereas we engineer a prescribed spectrum on a fixed compact covering manifold with no non-compact limit involved.
\item The ``graph-like manifold'' spectral convergence literature
(see, e.g., Exner--Post \cite{ExnerPost}
and Grieser \cite{Grieser}) treats thin tubular neighborhoods of metric graphs. Our corridors
are related but differ in that the ambient manifold has \emph{constant
curvature} and the ``thickening'' comes from the fixed hyperbolic block~$F$
rather than from shrinking a transverse profile.
\item \v{S}ahovic \cite{Sahovic} studied metric approximation of compact
metric spaces by hyperbolic manifolds with $|\kappa|\to\infty$, in the
Gromov-Hausdorff sense. The present work addresses a \emph{spectral}
prescription problem rather than a metric approximation problem, and the
construction requires controlling eigenfunction behavior (via expander wiring
and min-max arguments) rather than only metric geometry.
\end{itemize}

\begin{remark}[Simple spectra only]
We restrict throughout to \emph{strictly increasing} (simple) target lists. This is not a genuine limitation for an $\varepsilon$-approximate theorem: if the desired target has multiplicities (e.g., $\{0,1,1,3\}$), one may apply Theorem~\ref{thm:main} to a nearby simple list (e.g., $\{0,1-\delta,1+\delta,3\}$ with $\delta<\varepsilon/2$) with tolerance $\varepsilon/2$, and the triangle inequality gives $|\lambda_k-\lambda_k^{\mathrm{orig}}|<\varepsilon/2+\delta<\varepsilon$ for the original degenerate target. The simple-spectrum hypothesis is imposed for notational convenience and because it simplifies the continuum construction. Whether the covering construction in $d\ge 3$ can produce \emph{exact} prescribed multiplicities under constant curvature is a separate and interesting open problem.
\end{remark}

\begin{remark}[Countability obstruction to exact prescription in $d\ge 3$]
\label{rem:countability}
The ``approximate'' nature of Theorem~\ref{thm:main} is not a limitation of the method but a necessary feature of the problem. Exact prescription of a single eigenvalue is easy by curvature rescaling. However, in $d\ge 3$, exact \emph{simultaneous} prescription of a generic $n$-tuple of eigenvalues ($n\ge 2$) cannot hold as a universal theorem. The reason is as follows. By Mostow Rigidity, a closed hyperbolic $d$-manifold is determined up to isometry by its fundamental group. Since fundamental groups of closed manifolds are finitely presented, there are only countably many closed hyperbolic $d$-manifolds (up to isometry) at curvature $\kappa=-1$. Allowing $\kappa$ to vary merely traces out a $1$-parameter ray in spectral space for each topology, so the set of all achievable spectra is a countable union of rays in $\mathbb{R}^n$. For $n\ge 2$, such a set has Lebesgue measure zero, making exact prescription of even two independent eigenvalues set-theoretically impossible for a generic target. The $\varepsilon$-approximation is therefore sharp as a structural statement. (Nevertheless, for specific algebraic targets-e.g., whether $\{0,1,3\}$ can be exactly achieved at some $\kappa=-1$-the question remains open.)
\end{remark}

\subsection*{Organization}
Section~\ref{sec:discrete} recalls the discrete inverse spectral theorem.
Section~\ref{sec:surfaces} treats the surface case as a warm-up.
Section~\ref{sec:higher} contains the higher-dimensional covering construction
and spectral reduction. Section~\ref{sec:proof} assembles the proof of
Theorem~\ref{thm:main}. Section~\ref{sec:obstruction} records the universal
eigenvalue obstructions at $\kappa=-1$. Section~\ref{sec:ratios} proves the
eigenvalue ratio corollary. Section~\ref{sec:example} gives an explicit worked
example.

\section{Spectral Universality of Discrete Complete Graphs}
\label{sec:discrete}

We use the inverse spectral theory of discrete weighted graph Laplacians on complete graphs.

Let $G=(\mathcal V,E)$ be a finite connected graph with $N=|\mathcal V|$ vertices.
Fix strictly positive vertex measures $\nu_i>0$ for $i\in\mathcal V$ and symmetric edge weights
$w_{ij}=w_{ji}>0$ for $\{i,j\}\in E$. The weighted combinatorial Laplacian $L_G$ acts on
$f\in\mathbb{R}^{\mathcal V}$ by
\begin{equation}
\label{eq:discrete_laplacian}
    (L_G f)(i) = \frac{1}{\nu_i} \sum_{j\sim i} w_{ij} (f(i)-f(j)).
\end{equation}
It is self-adjoint on $\ell^2(\mathcal V,\nu_i)$ and has spectrum
$0=\mu_0<\mu_1\le \cdots \le \mu_{N-1}$.

\begin{lemma}[Colin de Verdi\`ere {\cite[\S 4, Theor\`eme~1]{CdV1988}}]
\label{lem:cdv_complete}
Let $G=K_N$ be the complete graph on $N\ge 2$ vertices with constant vertex measure $\nu_i\equiv 1$.
For any strictly increasing target list $0<\mu_1^*<\cdots<\mu_{N-1}^*$ there exist
strictly positive symmetric edge weights $w_{ij}=w_{ji}>0$ (for all $\{i,j\}\in E(K_N)$) such that the weighted combinatorial Laplacian $L_{K_N}$ defined by~\eqref{eq:discrete_laplacian} has non-zero spectrum
$(\mu_1^*,\dots,\mu_{N-1}^*)$. In particular, the prescribed eigenvalues are simple.
\end{lemma}

\begin{proof}
Colin de Verdi\`ere \cite[\S 4, Theor\`eme~1]{CdV1988} proves that for the complete graph $K_N$ ($N\ge 2$) with any prescribed positive vertex measure, one can realize any strictly increasing list $0<\mu_1^*<\cdots<\mu_{N-1}^*$ as the non-zero spectrum of a weighted combinatorial Laplacian with strictly positive edge weights on every edge of $K_N$. (Note: the result concerns the pure Laplacian class~\eqref{eq:discrete_laplacian}, with zero row sums and positive edge weights; no additional vertex potentials or Schr\"odinger-type diagonal perturbations are needed. The 1987 paper \cite[\S 2a]{CdV1987} invokes this discrete construction explicitly, pointing to \cite[\S 4]{CdV1988} for the proof.) In the constant-measure case $\nu_i\equiv 1$, this directly gives the Laplacian $L_{K_N}$ of~\eqref{eq:discrete_laplacian} with $w_{ij}>0$ for all $\{i,j\}$.
\end{proof}

\begin{remark}[Scaling with constant vertex measure]
\label{rem:constant_measure}
If $\nu_i\equiv \nu$ is a positive constant, then $L_G$ is simply scaled by $\nu^{-1}$ relative to the
case $\nu\equiv 1$. Hence Lemma~\ref{lem:cdv_complete} immediately implies the corresponding statement
for any constant vertex measure.
\end{remark}

\section{The Surface Case (\texorpdfstring{$d=2$}{d=2}): Degenerating Collars}
\label{sec:surfaces}

We recall the standard hyperbolic collar degeneration that produces a discrete graph limit. Although the full inverse spectral construction uses the complete graph $K_N$ (via Lemma~\ref{lem:cdv_complete}), the spectral convergence result itself applies to \emph{any} weighted graph realized by a system of disjoint simple closed geodesics on the surface.

Fix $N\ge 4$. (This covers all target list lengths $n\ge 1$, since the proof of Theorem~\ref{thm:main} takes $N=\max(n+1,4)\ge 4$ and pads the discrete spectrum with strictly increasing extra eigenvalues above the target window when $N>n+1$.)
Let $\Sigma$ be a closed orientable surface of genus $\gamma=1+\frac{N(N-3)}{2}$. Choose a collection $\mathcal{A}=\{\gamma_e\}_{e\in E(K_N)}$ of $\binom{N}{2}$ pairwise disjoint simple closed curves on $\Sigma$ whose dual graph is $K_N$: cutting $\Sigma$ along $\mathcal{A}$ produces $N$ connected components $X_1,\dots,X_N$, each of genus $0$ with $N-1$ boundary curves. (Such a curve system exists: begin with $N$ disjoint spheres with $N-1$ holes and connect them by $\binom{N}{2}$ tubes.) Fix a complete hyperbolic metric $g_0$ on $\Sigma$ in which all curves $\gamma_e$ are geodesic, and equip $\Sigma$ with the Fenchel-Nielsen coordinates associated to a pants decomposition that includes $\mathcal{A}$.

For each $\delta>0$, let $\Sigma_\delta$ be the hyperbolic surface obtained by setting
\begin{equation}
\label{eq:core_length}
    \ell(\gamma_e)=\pi\delta  w_e
\end{equation}
for prescribed edge weights $w_e>0$, while \emph{holding all remaining Fenchel-Nielsen coordinates fixed} (both the lengths of curves in the pants decomposition not belonging to $\mathcal{A}$ and all twist parameters). By Gauss-Bonnet, the vertex pieces $X_{v,\delta}$ (whose boundary lengths vary with $\delta$) satisfy $\operatorname{Area}(X_{v,\delta})=2\pi(N-3)$ independently of $\delta$. The curves in $\mathcal{A}$ are disjoint simple closed geodesics on $\Sigma_\delta$ for every $\delta>0$, and the Collar Theorem \cite[Theorems~4.1.1 and~4.1.6]{Buser1992} provides embedded disjoint collar neighborhoods around each.

The key analytic input is that the collar conductance is asymptotically $\ell_e/\pi$ as $\ell_e\to 0$; see \cite[\S 8.1]{Buser1992}.

\begin{theorem}[Burger {\cite{Burger1988,Burger1990}}; Buser {\cite[\S 8]{Buser1992}}]
\label{thm:burger}
Let $\Sigma$ be a closed orientable hyperbolic surface, $\mathcal{A}=\{\gamma_e\}_{e\in E(G)}$ a collection of pairwise disjoint simple closed geodesics on $\Sigma$ whose dual graph is a connected graph~$G$ on $N$ vertices, and $V_v:=\operatorname{Area}(X_v)$ the area of the $v$th vertex piece. Let $\Sigma_\delta$ be the family obtained by setting $\ell(\gamma_e)=\pi\delta  w_e$ for prescribed edge weights $w_e>0$, while holding all remaining Fenchel-Nielsen coordinates fixed, and let $\mu_k$ be the eigenvalues of the discrete Laplacian $L_{G}$ with edge weights $w_e$ and vertex measures $V_v$.
Then as $\delta\to 0$,
\[
\lambda_0(\Sigma_\delta)=0,\qquad\lambda_k(\Sigma_\delta)=\delta \mu_k (1+o(1))\qquad \text{for }k=1,\dots,N-1,
\]
and there exists a constant $C_0>0$, independent of $\delta$, such that
\[
\lambda_k(\Sigma_\delta)\ge C_0\qquad \text{for all }k\ge N.
\]
\end{theorem}

\begin{remark}
The key hypothesis underlying both the eigenvalue asymptotics and the uniform gap $C_0$ is the non-degeneration of the complementary pieces: because the Fenchel-Nielsen lengths outside $\mathcal{A}$ are held fixed, no other geodesic on $\Sigma_\delta$ becomes short, and the Neumann spectral gap of each vertex piece $X_{v,\delta}$ (after removing controlled collar neighborhoods at the pinching curves) remains uniformly positive; see \cite[\S\S 4--5]{Burger1990} for the quantitative bounds. The eigenvalue asymptotics $\lambda_k\sim\delta\mu_k$ are established in \cite[Theorems~1.1--1.2]{Burger1990}; the uniform lower bound $C_0$ for $k\ge N$ follows from the uniform Neumann gap of the vertex pieces \cite[\S\S 4--5]{Burger1990} together with the collar energy estimates in \cite[\S 8.1]{Buser1992}. The $K_N$ construction with equal vertex areas $V_v=2\pi(N-3)$ described above is the primary instance used in Section~\ref{sec:proof}; the worked example in Section~\ref{sec:example} uses $G=P_3$ with unequal vertex areas.
\end{remark}

\begin{remark}
\label{rem:general_graph}
The proof of Theorem~\ref{thm:burger} uses only two properties of the pinching family: that each collar conductance is asymptotic to $\ell_e/\pi$ as $\ell_e\to 0$ \cite[\S 8]{Buser1992}, and that the Neumann spectral gap of each vertex piece remains uniformly positive \cite[\S\S 4--5]{Burger1990}; the latter is ensured by the non-degeneration condition (all non-pinched Fenchel-Nielsen coordinates remain bounded). The case $G=K_N$ is used in Section~\ref{sec:proof} for the full inverse spectral theorem (where Lemma~\ref{lem:cdv_complete} requires a complete graph); the worked example in Section~\ref{sec:example} uses $G=P_3$ for its algebraic tractability.
\end{remark}
\section{Higher Dimensions (\texorpdfstring{$d\ge 3$}{d>=3}): Heavy-Vertex Coverings and Discrete Limits}
\label{sec:higher}

In dimensions $d\ge 3$, we construct hyperbolic covering manifolds whose macroscopic low-energy dynamics converge to a prescribed \emph{discrete} graph Laplacian via a ``heavy vertex / long corridor'' regime.

\subsection{A base manifold surjecting onto a free group}

By a theorem of Millson \cite{Millson1976}, for each $d\ge 4$ there exist closed arithmetic hyperbolic $d$-manifolds $M_0$ with $b_1(M_0)>0$. Millson's construction produces arithmetic lattices of simplest type in $\mathrm{SO}_0(d,1)$, defined by admissible quadratic forms over totally real number fields, with non-vanishing first cohomology.
By Lubotzky \cite{Lubotzky1996}, arithmetic lattices of simplest type in $\mathrm{SO}(d,1)$ are \emph{large} for all $d\ge 3$: some finite-index subgroup surjects onto the free group~$F_2$. (Lubotzky's proof uses the totally geodesic codimension-$1$ submanifolds furnished by the arithmetic structure; the conclusion $b_1>0$ is a consequence of largeness, not a hypothesis. In particular, for $d=3$ the largeness theorem applies directly to cocompact arithmetic Kleinian groups of simplest type, which always contain totally geodesic surfaces by their arithmetic construction, even though Millson's original construction is stated for $d\ge 4$.)

Fix a target length $n$ and choose an odd integer $N\ge \max(n+1,5)$; set $D=(N-1)/2$.
By Schreier theory (see, e.g., \cite[Ch.~I, Prop.~3.9]{LyndonSchupp}), $F_2$ contains $F_D$ as a finite-index subgroup, hence after passing to a finite cover we obtain:

\begin{proposition}
\label{prop:base}
For each $d\ge 3$ and each $D\ge 2$ there exists a closed hyperbolic $d$-manifold $B$ and a surjection $\pi_1(B)\twoheadrightarrow F_D$.
\end{proposition}

\begin{remark}[Non-constructiveness]
\label{rem:nonconstructive}
The base manifold $B$ is obtained existentially via Millson-Lubotzky. The construction is non-constructive at its foundation: not only is $B$ existential, but the fundamental block $F$ (obtained by cutting $B$ along smooth hypersurfaces), the analytic constants $C_P$, $C_i$, $C_\chi$, and the expander wiring (via the probabilistic method) are all non-explicit. In particular, the effective constants in the $O(\cdot)$ error estimates throughout Section~4 depend on the geometry of~$F$, which is in turn determined by $B$ and the cut system. For $d=3$, explicit \emph{closed} arithmetic hyperbolic $3$-manifolds are available from cocompact arithmetic Kleinian groups of simplest type (e.g., unit groups of orders in quaternion algebras over number fields with exactly one complex place, ramified at all real places; see \cite[\S\S 8.1--8.2, 9.5]{MaclachlanReid2003} for the general framework). Such groups contain totally geodesic surfaces by their arithmetic construction, so Lubotzky's largeness theorem applies; however, making the full Section~4 construction effective would additionally require an explicit cut system and explicit Schreier graphs, which we do not pursue. (We caution that the Bianchi groups $\mathrm{PSL}_2(\mathcal{O}_d)$, often cited as canonical examples of arithmetic hyperbolic $3$-manifold groups, are \emph{non-cocompact}: they yield cusped manifolds, not the closed base manifold~$B$ required here.) For $d\ge 4$, no explicit base manifold $B$ satisfying our requirements (closed, hyperbolic, with $\pi_1(B)\twoheadrightarrow F_D$) is known to the author.
\end{remark}

Let $X_D$ denote the bouquet of $D$ circles with basepoint $v_0$, which is a $K(F_D,1)$ space ($\pi_1(X_D)\cong F_D$ and $\pi_k(X_D)=0$ for $k\ge 2$). Write $e_1,\dots,e_D$ for the open petals.

\begin{proposition}[Connected cut system]
\label{prop:cut_system}
For $d\ge 3$, there exist pairwise disjoint connected smooth closed two-sided hypersurfaces
\[
H_1,\dots,H_D\subset B
\]
such that the complement $B\setminus(H_1\cup\cdots\cup H_D)$ is connected.
\end{proposition}

\noindent\textbf{Convention.}
We write $F^\circ:=B\setminus(H_1\cup\cdots\cup H_D)$ for the open complement, and $F$ for the compact manifold with boundary obtained by cutting $B$ open along the $H_i$. Thus $\operatorname{int}(F)=F^\circ$ and $\partial F=H_1^+\cup H_1^-\cup\cdots\cup H_D^+\cup H_D^-$, where $H_i^\pm$ are the two copies of $H_i$ arising from the two sides of the cut. Unless otherwise stated, $F$ henceforth denotes the compact block.
Crucially, because the hypersurfaces $H_1,\dots,H_D$ are pairwise disjoint, $\partial F$ is a disjoint union of smooth closed hypersurfaces and $F$ has no corners or codimension-$2$ boundary singularities. Standard Sobolev trace theorems and elliptic regularity estimates therefore apply uniformly on~$F$; the analytic constants $C_{\mathrm{tr}}$, $C_{\mathrm{bl}}$, and $C_P$ appearing in the sequel depend only on the fixed Riemannian geometry of~$F$.

\begin{proof}
\emph{Step 1: Auxiliary map and its fibers.}
Since $X_D$ is a $K(F_D,1)$, the surjection $\pi_1(B)\twoheadrightarrow F_D$ from Proposition~\ref{prop:base} is realized by a continuous map $g:B\to X_D$ inducing the given homomorphism on $\pi_1$. For each petal $e_j$, choose an open subarc $U_j\subset e_j$. By smooth approximation, we may homotope $g$ so that $g$ is smooth on $g^{-1}(U_j)$ for each~$j$; Sard's theorem then provides regular values $q_j\in U_j$ at which $g$ is transverse. Since $g_*$ is surjective, $g$ maps some loop in $B$ to a loop of degree $1$ on the $j$th circle; such a map is surjective onto~$e_j$, so in particular $q_j\in g(B)$ and each fiber $g^{-1}(q_j)$ is nonempty. Each fiber $g^{-1}(q_j)$ is a smooth closed two-sided codimension-$1$ submanifold of $B$, possibly disconnected. We co-orient each component $\Sigma$ of $g^{-1}(q_j)$ by pulling back the positive orientation of~$e_j$; the bicollar $N_\Sigma\cong\Sigma\times[-1,1]$ is chosen compatible with this co-orientation. Let
\[
\mathcal{S}:=\bigl\{\text{connected components of }g^{-1}(q_j):j=1,\dots,D\bigr\}.
\]
The elements of $\mathcal{S}$ are finitely many pairwise disjoint connected smooth closed two-sided hypersurfaces. Each $\Sigma\in\mathcal{S}$ inherits a \emph{color} $\mathrm{col}(\Sigma):=j$ from its parent fiber.

\emph{Step 2: Full dual graph.}
Cut $B$ along every $\Sigma\in\mathcal{S}$ and form the dual graph $\Gamma$: vertices of~$\Gamma$ are the connected components of
$B\setminus\bigcup_{\Sigma\in\mathcal{S}}\Sigma$,
and each $\Sigma\in\mathcal{S}$ contributes one edge $e_\Sigma$ joining the two adjacent complementary regions (a loop if both sides lie in the same region). The graph $\Gamma$ is finite and connected (since $B$ is).

\emph{Step 3: Surjectivity of $\eta_*$ on homology.}
Define the graph map $\eta:\Gamma\to X_D$ by sending each vertex of $\Gamma$ to $v_0$ and each edge $e_\Sigma$ (with $\mathrm{col}(\Sigma)=j$) around the $j$th loop of~$X_D$. Let $Q:B\to\Gamma$ be the quotient map associated to the cut decomposition: it collapses each complementary region to the corresponding vertex and maps each bicollar of a component $\Sigma\in\mathcal{S}$ to the corresponding dual edge~$e_\Sigma$.

We claim that $\eta_*:H_1(\Gamma;\mathbb{Z})\to H_1(X_D;\mathbb{Z})\cong\mathbb{Z}^D$ is surjective.

For each $j\in\{1,\dots,D\}$, surjectivity of $g_*:\pi_1(B)\twoheadrightarrow F_D$ provides a smooth loop $\gamma_j$ in $B$ with $[g\circ\gamma_j]=a_j$ (the $j$th free generator) in $F_D$. We may assume $\gamma_j$ is transverse to every $\Sigma\in\mathcal{S}$. Since $[g\circ\gamma_j]=a_j$, the algebraic intersection number of $\gamma_j$ with $g^{-1}(q_k)$ equals the exponent-sum of $a_k$ in $a_j$, which is $\delta_{jk}$.

Since $\gamma_j$ is a closed loop, the quotient $Q(\gamma_j)$ is a cycle $c_j$ in $\Gamma$. For each color $k$, the algebraic count of color-$k$ edges traversed by $c_j$ equals the algebraic intersection of $\gamma_j$ with $\bigcup\{\Sigma\in\mathcal{S}:\mathrm{col}(\Sigma)=k\}=g^{-1}(q_k)$, which is $\delta_{jk}$. Under $\eta$, each color-$k$ edge maps to a loop around the $k$th petal; therefore the $k$th coordinate of $[\eta_*(c_j)]$ in $H_1(X_D;\mathbb{Z})\cong \mathbb{Z}^D$ is $\delta_{jk}$. Since $\{\eta_*(c_j):j=1,\dots,D\}$ are the standard basis vectors of $\mathbb{Z}^D$, the map $\eta_*:H_1(\Gamma;\mathbb{Z})\to H_1(X_D;\mathbb{Z})$ is surjective.

\emph{Step 4: Selecting a colored transversal via the cycle space.}
For each edge~$e$ of color~$j$, orient~$e$ so that $\eta$ traverses the $j$th petal of~$X_D$ positively on~$e$. For an oriented edge~$e$, let $e^*\in C^1(\Gamma;\mathbb{R})$ denote the dual edge cochain.
Write $E_j\subset E(\Gamma)$ for the set of edges of color~$j$, and let $Z_1:=H_1(\Gamma;\mathbb{R})$ denote the cycle space.
Since $\eta_*:H_1(\Gamma;\mathbb{Z})\to H_1(X_D;\mathbb{Z})\cong\mathbb{Z}^D$ is surjective (Step~3), the induced map on real homology $\eta_*:Z_1\to H_1(X_D;\mathbb{R})\cong\mathbb{R}^D$ is also surjective. For each color~$j$, the \emph{color-$j$ evaluation} $v_j\in Z_1^*=H^1(\Gamma;\mathbb{R})$ is the linear functional that counts (with sign) the algebraic number of color-$j$ edges traversed by a cycle. Since $\eta_*$ is surjective on $Z_1$, the functionals $v_1,\dots,v_D$ are linearly independent on~$Z_1$.

Each $v_j$ decomposes as $v_j=\sum_{e\in E_j} e^*\big|_{Z_1}$. Consider the exterior product in $\Lambda^D(Z_1^*)$:
\[
v_1\wedge\cdots\wedge v_D
=
\sum_{(a_1,\dots,a_D)\in E_1\times\cdots\times E_D}
a_1^*\big|_{Z_1}\wedge\cdots\wedge a_D^*\big|_{Z_1}.
\]
Since $v_1,\dots,v_D$ are linearly independent, this wedge product is nonzero, so at least one summand is nonzero. Choose a $D$-tuple $(a_1,\dots,a_D)$ with $a_i\in E_i$ for which $a_1^*|_{Z_1},\dots,a_D^*|_{Z_1}$ are linearly independent. Let $H_i\in\mathcal{S}$ be the hypersurface component corresponding to~$a_i$.

\emph{Step 5: Connected complement.}
We claim that $\Gamma\setminus\{a_1,\dots,a_D\}$ is connected. Suppose not: then there is a nontrivial vertex partition $(S,S^c)$ with no edges of $\Gamma\setminus\{a_1,\dots,a_D\}$ crossing between $S$ and $S^c$. The coboundary $\delta(\mathbf{1}_S)\in C^1(\Gamma;\mathbb{R})$ (assigning $\pm 1$ to edges crossing the partition and $0$ elsewhere) is then supported on $\{a_1,\dots,a_D\}$, giving a nontrivial linear combination $\sum_i \alpha_i a_i^*=\delta(\mathbf{1}_S)$ with $\alpha_i\in\{-1,0,1\}$. Since $\delta(\mathbf{1}_S)$ lies in the cut space $B^1(\Gamma;\mathbb{R})=Z_1^\perp$, we obtain $\sum_i \alpha_i a_i^*\big|_{Z_1}=0$, contradicting the linear independence of $a_1^*|_{Z_1},\dots,a_D^*|_{Z_1}$.

Hence $\Gamma\setminus\{a_1,\dots,a_D\}$ is connected. Write $\Gamma_{\mathrm{cut}}$ for the dual graph of the decomposition of $B$ by only $H_1,\dots,H_D$; it is obtained from~$\Gamma$ by contracting every non-loop edge of $\Gamma\setminus\{a_1,\dots,a_D\}$ and deleting every loop of $\Gamma\setminus\{a_1,\dots,a_D\}$. Since $\Gamma\setminus\{a_1,\dots,a_D\}$ is connected, all vertices merge into a single point, and the $D$ chosen edges $a_1,\dots,a_D$ survive as loops at that point. Thus $\Gamma_{\mathrm{cut}}$ is a bouquet of $D$~circles; in particular, it has exactly one vertex, so $F^\circ=B\setminus(H_1\cup\cdots\cup H_D)$ is connected.
\end{proof}

\begin{lemma}[Bouquet map from the cut system]
\label{lem:bouquet_map}
For the hypersurfaces $H_1,\dots,H_D$ from Proposition~\ref{prop:cut_system}, there exists a continuous map $f:B\to X_D$, smooth on $f^{-1}(e_j)$ for each $j$, and interior points $p_i\in e_i$, such that:
\begin{enumerate}
\item[\emph{(i)}] $f^{-1}(p_i)=H_i$ for each $i=1,\dots,D$ (so each selected fiber is connected by construction);
\item[\emph{(ii)}] $f_*:\pi_1(B)\twoheadrightarrow F_D$ is surjective;
\item[\emph{(iii)}] $f(F^\circ)\subset X_D\setminus\{p_1,\dots,p_D\}$.
\end{enumerate}
\end{lemma}

\begin{proof}
Since the $H_i$ are smooth, closed, and two-sided, choose pairwise disjoint collar embeddings $c_i:H_i\times[-1,1]\hookrightarrow B$, with $c_i(y,0)=y$. Let $C_i:=c_i(H_i\times(-1,1))$ denote the open collar, and set $F_0:=B\setminus\bigcup_{i=1}^D C_i$. For sufficiently thin collars, $F_0$ is diffeomorphic to the compact block~$F$, hence connected.

For each $i$, choose a continuous map $\theta_i:[-1,1]\to\overline{e_i}$, smooth on $(-1,1)$, such that $\theta_i(-1)=\theta_i(1)=v_0$ and $\theta_i|_{(-1,1)}:(-1,1)\to e_i$ is a diffeomorphism (so that the induced loop traverses the $i$th petal exactly once). Set $p_i:=\theta_i(0)\in e_i$. Define
\[
f(x):=
\begin{cases}
v_0, & x\in F_0,\\[4pt]
\theta_i(t), & x=c_i(y,t)\in C_i.
\end{cases}
\]
This is continuous: on the boundary of each collar, $\theta_i(\pm 1)=v_0$, matching the value on~$F_0$.

\emph{Fiber structure.} Since $\theta_i|_{(-1,1)}$ is a diffeomorphism, $\theta_i^{-1}(p_i)=\{0\}$ and $p_i$ is a regular value, so $f^{-1}(p_i)=c_i(H_i\times\{0\})=H_i$. Moreover, $f(F_0)=\{v_0\}$ and $f(C_i)\subset\overline{e_i}$, so $f^{-1}(p_j)\cap C_i=\varnothing$ for $j\neq i$.

\emph{Surjectivity.} Fix a basepoint $b_0\in F_0$. Since $F_0$ is connected, for each $i$ there exists a path $\alpha_i$ in $F_0$ from $b_0$ to the point $c_i(y_i,-1)$ on one side of the $i$th collar, for some $y_i\in H_i$. Let $\beta_i$ be the path $t\mapsto c_i(y_i,t)$ for $t\in[-1,1]$, which crosses the collar from side to side. Let $\alpha_i'$ be a return path in $F_0$ from $c_i(y_i,1)$ back to $b_0$ (such a path exists because $c_i(y_i,\pm 1)\in F_0$ and $F_0$ is connected). The concatenation $\gamma_i:=\alpha_i\cdot\beta_i\cdot\alpha_i'$ is a loop in $B$ based at $b_0$. Under $f$, the segments $\alpha_i$ and $\alpha_i'$ map to the basepoint $v_0$, while $\beta_i$ maps via $\theta_i$ once around the $i$th petal. Hence $f_*([\gamma_i])$ is the $i$th free generator of $\pi_1(X_D,v_0)\cong F_D$. Since $i$ was arbitrary, $f_*$ is surjective.

\emph{Property (iii)} is immediate: $F^\circ=B\setminus\bigcup_i H_i$, so $f(F^\circ)\subset\{v_0\}\cup\bigcup_i(\overline{e_i}\setminus\{p_i\})\subset X_D\setminus\{p_1,\dots,p_D\}$.
\end{proof}

\begin{remark}[Pullback cover decomposition - the conceptual hinge]
\label{rem:pullback_decomposition}
The connected cut system and the resulting fiber structure are the central requirements for the higher-dimensional construction. By Lemma~\ref{lem:bouquet_map}(iii), $f(F^\circ)\subset X_D\setminus\{p_1,\dots,p_D\}$, which is contractible. Hence $f_*(\pi_1(F^\circ))=\{1\}$ in $F_D$. Let $S:=F_D/\Lambda_m$ denote the set of cosets (so $|S|=|F_D:\Lambda_m|$). The covering $X_m'\to X_D$ corresponding to $\Lambda_m$ restricts to a trivial covering over the contractible set $X_D\setminus\{p_1,\dots,p_D\}$. Since $f(F^\circ)$ lies in this set, the pullback covering $M_m\to B$ restricts to a \emph{trivial} covering over $F^\circ$: the preimage of $F^\circ$ in $M_m$ is a disjoint union of $|S|$ copies of~$F^\circ$, each mapping isometrically to~$F^\circ$. Cutting $M_m$ open along the lifts of the $H_i$, the closures of these components are $|S|$ copies of the compact block~$F$, and $M_m$ is obtained by gluing these copies along their boundary faces $H_i^\pm$ according to the Schreier graph $\Gamma_m$.

The connectedness of each $H_i$ is essential: it guarantees that each color contributes exactly two boundary faces $H_i^\pm$ to the block~$F$, giving the clean $2D$-face structure needed for the corridor and cluster construction in Sections~4.2-4.4.
\end{remark}

\noindent For the remainder of the paper, $\partial F = H_1^+\cup H_1^-\cup\cdots\cup H_D^+\cup H_D^-$ as in the Convention above.

\begin{remark}
For fixed $n$, the choice of $N$, hence $D$, hence the base manifold $B$ and block $F$, are fixed once and for all. Only the scaling parameter $m\to\infty$ will vary in the approximation.
\end{remark}

\subsection{Cell problem and effective conductances}

Fix a color $i\in\{1,\dots,D\}$. Consider the \emph{bi-infinite $i$-periodic chain} $\widetilde{S}_i$ obtained by gluing a $\mathbb{Z}$-indexed sequence of copies of the block $F$ end-to-end along the $i$-faces, and within each copy gluing $H_j^+$ to $H_j^-$ for every $j\neq i$ (so those faces become interior interfaces). The resulting $\widetilde{S}_i$ is a smooth non-compact manifold with an isometric deck transformation $T_i$ shifting by one block in the $i$-direction.

There exists a (unique up to additive constant) \emph{harmonic coordinate} $\chi_i$ on $\widetilde{S}_i$ satisfying the translation rule
\[
\chi_i\circ T_i = \chi_i + 1.
\]

\emph{Existence and regularity via Hodge theory.}
Recall that the quotient $\widetilde{S}_i/\langle T_i\rangle$ is canonically isometric to the base manifold~$B$ (gluing the two $i$-faces of $F$ with the shift $T_i$ and the $j$-faces for $j\neq i$ reconstructs~$B$). Since $B$ is a closed orientable Riemannian manifold (orientability holds because the arithmetic lattices of simplest type in the Millson--Lubotzky construction lie in the identity component $\mathrm{SO}_0(d,1)$, which preserves orientation on~$\mathbb{H}^d$), Hodge theory provides a unique harmonic $1$-form $\omega_i$ on $B$ whose de~Rham class is the Poincar\'e dual of $[H_i]\in H_{d-1}(B;\mathbb{R})$.

We verify that $[\omega_i]\neq 0$: since $H_i$ is non-separating (the complement $F^\circ=B\setminus(H_1\cup\cdots\cup H_D)$ is connected by Proposition~\ref{prop:cut_system}, and $\bigcup_{j\neq i}H_j$ is a finite union of smooth closed hypersurfaces, hence closed with empty interior in $B\setminus H_i$; therefore $F^\circ$ is dense in $B\setminus H_i$, and $B\setminus H_i$ is connected as the closure of a connected set), the homology class $[H_i]$ is nontrivial, and so is its Poincar\'e dual.

The covering $\pi:\widetilde{S}_i\to B$ is the regular $\mathbb{Z}$-cover
classified by the homomorphism
\[
\alpha_i:\pi_1(B)\to\mathbb{Z},\qquad \alpha_i([\gamma])=[\gamma]\cdot [H_i],
\]
given by algebraic intersection with $H_i$: in the cut-and-stack model,
crossing $H_i$ changes the block index by $\pm 1$, while crossing any
$H_j$ with $j\neq i$ leaves the index unchanged. Hence
$\pi_1(\widetilde{S}_i)=\ker\alpha_i$.
Since $[\omega_i]=\mathrm{PD}[H_i]$ represents the same cohomology class as $\alpha_i$
under the de~Rham isomorphism $H^1(B;\mathbb{R})\cong\mathrm{Hom}(\pi_1(B),\mathbb{R})$,
we have $\pi^*[\omega_i]=0$ in $H^1(\widetilde{S}_i;\mathbb{R})$, so
$\tilde\omega_i:=\pi^*\omega_i$ is exact. Choose $\chi_i$ with
$d\chi_i=\tilde\omega_i$.

For the deck generator $T_i$, we have
$d(T_i^*\chi_i-\chi_i)=T_i^*d\chi_i-d\chi_i=0$,
so $T_i^*\chi_i-\chi_i$ is a constant on the connected manifold
$\widetilde{S}_i$. If $\tilde\gamma$ is any path from $x$ to $T_i x$,
its projection $\gamma=\pi\circ\tilde\gamma$ satisfies
$\alpha_i([\gamma])=1$, hence
\[
\chi_i(T_i x)-\chi_i(x)
=\int_{\tilde\gamma} d\chi_i
=\int_\gamma \omega_i
=\langle[\omega_i],[\gamma]\rangle = \alpha_i([\gamma])=1.
\]
Therefore $\chi_i\circ T_i=\chi_i+1$. Since $d\chi_i=\tilde\omega_i$ is harmonic ($\delta\tilde\omega_i=0$), we have $\Delta\chi_i=0$ on $\widetilde{S}_i$. In particular, $\chi_i$ is smooth, and on each block $F$ of the chain the conormal transmission condition
\begin{equation}
\label{eq:transmission}
\partial_{n^+}\chi_i = -\partial_{n^-}\chi_i
\end{equation}
holds across every interior interface (simply because $\chi_i$ is smooth across these interfaces, which are lifts of the smooth hypersurfaces $H_j\subset B$).
We fix the remaining additive constant by requiring $\frac{1}{|H_i^-|}\int_{H_i^-}\chi_i  dA = 0$, so that $\chi_i|_{H_i^-}$ has mean zero and $\chi_i|_{H_i^+}$ has mean $1$.

Define the \emph{effective conductance constant} by the energy per period:
\[
C_i := \int_{F} |\nabla \chi_i|^2  dV >0,
\]
where $F$ is identified with a fundamental domain for $T_i$. Positivity is immediate: if $C_i=0$ then $\chi_i$ would be constant on each period, contradicting the jump condition.

\begin{remark}[Homogenization viewpoint]
\label{rem:hodge}
The harmonic coordinate $\chi_i$ solves the \emph{cell problem} of periodic homogenization for the Laplacian on the $\mathbb{Z}$-periodic manifold $\widetilde{S}_i$: it minimizes the Dirichlet energy per period among all functions satisfying $\chi\circ T_i=\chi+1$ (compare \cite[Ch.~1]{BensoussanLionsPapanicolaou} for the Euclidean analogue). This variational characterization is equivalent to the Hodge-theoretic construction above but generalizes more readily to non-uniform corridor scalings.
\end{remark}

The flux normalization is obtained by integrating by parts on one period: since $\Delta\chi_i=0$ on the interior of~$F$,
\[
C_i = \int_F |\nabla\chi_i|^2  dV = \int_{\partial F} \chi_i  \partial_n\chi_i  dA.
\]
The contributions from the $j$-faces ($j\neq i$) cancel in pairs: on each identified pair $H_j^+\cong H_j^-$, the traces of $\chi_i$ match (periodic boundary condition) and the conormal derivatives satisfy $\partial_{n^+}\chi_i=-\partial_{n^-}\chi_i$ (transmission condition~\eqref{eq:transmission}), so the two boundary integrals cancel. On the two $i$-faces $H_i^\pm$, the traces differ by exactly $1$ (by the jump condition), and the outward normal fluxes satisfy $\partial_{n^+}\chi_i=-\partial_{n^-}\chi_i$ after transport (again by~\eqref{eq:transmission}). Hence
\begin{equation}
\label{eq:flux_normalization}
\int_{H_i^+} \partial_n \chi_i  dA = C_i,
\qquad
\int_{H_i^-} (-\partial_n \chi_i)  dA = C_i.
\end{equation}

\begin{lemma}[Corridor bounds and series scaling]
\label{lem:corridor}
Let $W_{i,K}$ be a finite corridor chain consisting of $K$ consecutive $i$-blocks, with all faces $j\neq i$ sealed by the same gluing as in $\widetilde{S}_i$, so that $\partial W_{i,K}$ consists of exactly two end faces $\Sigma_i^-$ and $\Sigma_i^+$ (each isometric to~$H_i^\pm$; we use the distinct notation $\Sigma$ to distinguish corridor end faces from single-cell boundary faces).

For $u\in H^1(W_{i,K})$ define the \emph{normalized flux averages} on the two ends by
\[
\langle u\rangle_- := \frac{1}{C_i}\int_{\Sigma_i^-} u (-\partial_n\chi_i) dA,
\qquad
\langle u\rangle_+ := \frac{1}{C_i}\int_{\Sigma_i^+} u (\partial_n\chi_i) dA.
\]
(These are normalized weights of total mass $1$ by \eqref{eq:flux_normalization}; they need not be pointwise nonnegative.)

Then:
\begin{enumerate}
\item[(i)] (\emph{Lower bound}) For every $u\in H^1(W_{i,K})$,
\begin{equation}
\label{eq:corridor_lower_bound}
\int_{W_{i,K}} |\nabla u|^2 \ge \frac{C_i}{K} \bigl(\langle u\rangle_+ - \langle u\rangle_-\bigr)^2.
\end{equation}
\item[(ii)] (\emph{Upper bound for constant endpoint data}) For any constants $a,b\in\mathbb{R}$ there exists $u_{a,b}\in H^1(W_{i,K})$ whose traces are \emph{constant} on the end faces, $u_{a,b}\equiv a$ on $\Sigma_i^-$ and $u_{a,b}\equiv b$ on $\Sigma_i^+$, and such that
\begin{equation}
\label{eq:corridor_upper_bound}
\int_{W_{i,K}} |\nabla u_{a,b}|^2
\le
\frac{C_i}{K}(a-b)^2 + \frac{C_{\mathrm{bl}}}{K^2}(a-b)^2,
\end{equation}
where $C_{\mathrm{bl}}>0$ depends only on the fixed block geometry (in particular, is independent of $K$).
\end{enumerate}
Consequently, the minimal Dirichlet energy with constant endpoint values satisfies
\[
\inf\Bigl\{\int_{W_{i,K}}|\nabla u|^2: u|_{\Sigma_i^-}\equiv a,\ u|_{\Sigma_i^+}\equiv b\Bigr\}
=
\frac{C_i}{K}(a-b)^2 + O\left(\frac{(a-b)^2}{K^2}\right).
\]
\end{lemma}

\begin{proof}
\textbf{(i)} Let $\chi$ denote the restriction of $\chi_i$ to the finite corridor $W_{i,K}$. By the cell-problem construction, $\chi$ is harmonic on each block interior ($\Delta\chi=0$) and satisfies the conormal transmission condition~\eqref{eq:transmission} across every interior glued interface; when integrating by parts on $W_{i,K}$ these two facts ensure that contributions from interior interfaces cancel pairwise, so only the two end faces contribute. Cauchy-Schwarz gives
\[
\int_{W_{i,K}} |\nabla u|^2
\ge
\frac{\bigl(\int_{W_{i,K}} \nabla u\cdot \nabla \chi\bigr)^2}{\int_{W_{i,K}} |\nabla \chi|^2}.
\]
The denominator equals $K C_i$ by definition of $C_i$ as energy per period. For the numerator, integration by parts on each block gives $\int_{W_{i,K}} \nabla u\cdot \nabla \chi = -\int_{W_{i,K}} u\Delta\chi + \int_{\partial W_{i,K}} u\partial_n\chi\,dA$ plus interior interface terms; since $\Delta\chi=0$ on each block interior and the transmission condition~\eqref{eq:transmission} cancels the interior interface contributions (the test function $u\in H^1(W_{i,K})$ has matching traces, and $\partial_{n^+}\chi=-\partial_{n^-}\chi$), only the two end faces survive:
\[
\int_{W_{i,K}} \nabla u\cdot \nabla \chi
=
\int_{\Sigma_i^+} u \partial_n\chi dA + \int_{\Sigma_i^-} u \partial_n\chi dA
=
C_i\bigl(\langle u\rangle_+ - \langle u\rangle_-\bigr),
\]
using the definitions of $\langle\cdot\rangle_\pm$. Substituting proves \eqref{eq:corridor_lower_bound}.

\smallskip
\textbf{(ii)} Consider the affine corrector profile $u_0 := a + \frac{b-a}{K}\chi$. Its energy is exactly
\[
\int_{W_{i,K}} |\nabla u_0|^2
=
\frac{(b-a)^2}{K^2}\int_{W_{i,K}}|\nabla\chi|^2
=
\frac{C_i}{K}(b-a)^2.
\]
However, the trace of $u_0$ on an end face need not be pointwise constant: it equals $a + \frac{b-a}{K}\chi|_{\Sigma_i^-}$ on the left and $a + \frac{b-a}{K}\chi|_{\Sigma_i^+}$ on the right. Since $\chi\circ T_i=\chi+1$, we may write $\chi|_{\Sigma_i^\pm}=c_\pm + \psi$ where $c_\pm$ are constants (with $c_+-c_-=K$) and $\psi$ is the \emph{same} mean-zero oscillatory part on both faces (identified via the chain periodicity). In particular, the boundary mismatch functions
\[
\delta_- := a - u_0|_{\Sigma_i^-},\qquad \delta_+ := b - u_0|_{\Sigma_i^+}
\]
satisfy $\|\delta_\pm\|_{L^2(\Sigma_i^\pm)} \le C |a-b|/K$ for a constant $C$ depending only on the fixed cell geometry. With our gauge convention (mean of $\chi_i$ over $\Sigma_i^-$ equals zero), the mismatches are given explicitly by
\[
\delta_- = -\frac{b-a}{K}\chi_i\big|_{\Sigma_i^-},\qquad \delta_+ = -\frac{b-a}{K}(\chi_i - K)\big|_{\Sigma_i^+},
\]
where $\chi_i|_{\Sigma_i^-}$ has mean zero and $(\chi_i-K)|_{\Sigma_i^+}$ also has mean zero (since $\chi_i|_{\Sigma_i^+}$ has mean $K$ by the jump condition). In particular, both $\delta_\pm$ are $O(|a-b|/K)$ multiples of fixed mean-zero smooth functions on the respective faces.

Choose smooth cutoff functions $\eta_\pm$ supported in a collar neighborhood of the corresponding end face $\Sigma_i^\pm$ within the first/last block, with $\eta_\pm\equiv 1$ on $\Sigma_i^\pm$ and $\eta_\pm\equiv 0$ outside the collar. Since the collar is chosen disjoint from all other boundary faces $H_j^\pm$ ($j\neq i$) of~$F$, the correction automatically respects the side-face identifications in the corridor manifold. Since $\chi_i$ is a smooth function defined on the entire chain, it provides an explicit smooth extension of its own boundary trace: on the first block we set $\delta_-^{\mathrm{ext}} := -\frac{b-a}{K}\chi_i$, and on the last block we set $\delta_+^{\mathrm{ext}} := -\frac{b-a}{K}(\chi_i - K)$ (where $\chi_i$ here denotes the restriction to the actual last block of the corridor, on which $\chi_i|_{\Sigma_i^+}$ has mean~$K$). The supports of $\eta_-$ and $\eta_+$ are disjoint: for $K>1$ they lie in different blocks, and for $K=1$ they are supported in disjoint collar neighborhoods of the two distinct boundary components $\Sigma_i^-$ and $\Sigma_i^+$ of the single block. In either case there is no cross-term between the two boundary-layer corrections. By the product rule,
\[
\int_{\text{(end block)}} |\nabla(\eta_\pm \delta_\pm^{\mathrm{ext}})|^2
\le
C_{\mathrm{bl}}\frac{(a-b)^2}{K^2},
\]
where $C_{\mathrm{bl}}$ depends only on $\|\nabla\chi_i\|_{L^2(F)}$, $\|\chi_i\|_{L^\infty(F)}$, and $\|\nabla\eta_\pm\|_{L^\infty}$, all of which are fixed geometric data independent of $K$. No abstract trace extension operators are needed. Define
\[
u_{a,b} := u_0 + \eta_- \delta_-^{\mathrm{ext}} + \eta_+ \delta_+^{\mathrm{ext}}.
\]
Then $u_{a,b}\equiv a$ on $\Sigma_i^-$ and $u_{a,b}\equiv b$ on $\Sigma_i^+$ by construction, and its energy differs from that of $u_0$ by at most the two boundary-layer energies plus a cross-term. For the cross-term: $\nabla u_0=\frac{b-a}{K}\nabla\chi$, and on a single boundary block $\int_{\text{block}}|\nabla\chi|^2=C_i$, so $\|\nabla u_0\|_{L^2(\text{block})}=|b-a|\sqrt{C_i}/K = O(|b-a|/K)$. The boundary layer extension satisfies $\|\nabla(\eta_\pm\delta_\pm^{\mathrm{ext}})\|_{L^2}=O(|b-a|/K)$ as well (by the same bound on $\delta_\pm$). Their Cauchy-Schwarz product is therefore $O(|b-a|/K)\times O(|b-a|/K)=O((b-a)^2/K^2)$, consistent with the $C_{\mathrm{bl}}/K^2$ bound. This proves \eqref{eq:corridor_upper_bound}.
\end{proof}

\begin{remark}
Since the hypersurfaces $H_1,\dots,H_D$ are pairwise disjoint (Proposition~\ref{prop:cut_system}), the block $F$ is a compact manifold with \emph{smooth} boundary consisting of $2D$ smooth faces $H_i^\pm$. Standard Sobolev trace theorems therefore apply on each face, and the constants $C_{\mathrm{tr}}, C_{\mathrm{bl}}$ depend only on the fixed geometry of~$F$.
\end{remark}

\subsection{Macroscopic networks with heavy vertices}

Fix a strictly increasing target spectrum $0=\lambda^*_0<\lambda^*_1<\cdots<\lambda^*_n$.
Choose odd $N\ge \max(n+1,5)$ and set $D=(N-1)/2$.
(We require $N\ge 5$ because for $N=3$ one has $D=1$, and $2D$-regular means $2$-regular, i.e.\ a union of disjoint cycles rather than an expander; the spectral-gap bound requires $D\ge 2$, i.e.\ degree $2D\ge 4$.)
By Walecki's theorem (see, e.g., \cite{Alspach2008}), the edges of $K_N$ decompose into $D$ edge-disjoint \emph{undirected} Hamiltonian cycles; we choose one of the two cyclic orientations for each cycle to obtain $D$ directed Hamiltonian cycles. This orientation ensures each vertex has exactly one incoming and one outgoing edge of each color, and we assign a \emph{color} $c(e)\in\{1,\dots,D\}$ to each edge accordingly.

Apply Lemma~\ref{lem:cdv_complete} to $K_N$ with vertex measures $V_v\equiv 1$ to obtain strictly positive edge weights $w_e^*>0$ such that the non-zero eigenvalues of $L_{K_N}$ are
\[
\mu_k = \lambda_k^* V_F \qquad (k=1,\dots,n),
\]
where $V_F:=\Vol(F)$ is the Riemannian volume of the fundamental block (and, if $N>n+1$, the remaining eigenvalues $\mu_{n+1},\dots,\mu_{N-1}$ are chosen to satisfy $\mu_k > (\lambda_n^*+1) V_F$ for all $k\ge n+1$; this is possible since Lemma~\ref{lem:cdv_complete} allows free prescription of a strictly increasing list on $K_N$).

Let $m\in\mathbb{N}$ be a large scaling parameter. Define:
\begin{itemize}
    \item \textbf{Heavy vertices:} each macroscopic vertex $v$ will be replaced by a \emph{cluster} of $m^3$ blocks $F$.
    \item \textbf{Long corridors:} each macroscopic edge $e$ will be replaced by a chain of
    \[
    K_e(m):=\left\lfloor \frac{m C_{c(e)}}{w_e^*}\right\rfloor
    \]
    blocks along color $c(e)$ (since $C_{c(e)}>0$ and $w_e^*>0$, we have $K_e(m)\ge 1$ for all $m$ sufficiently large).
\end{itemize}

\paragraph{Expander wiring inside clusters (with exposed ports).}
Inside each cluster we glue the $m^3$ blocks along their $2D$ faces so that (i) the internal gluing graph is a connected expander with uniform spectral gap, and (ii) for each color $i$ we expose exactly one incoming and one outgoing face to attach the corridors corresponding to the unique incoming/outgoing macroscopic edges of color $i$ incident to $v$.

Concretely, for each cluster we need a simple $2D$-regular graph on $m^3$ vertices that (a) has a uniform spectral gap, (b) has edge-connectivity $2D$, and (c) admits a decomposition into $D$ directed permutation color classes.

For (a) and (b): by Friedman's theorem \cite{Friedman2008}, a uniform random simple $2D$-regular graph on $n$ vertices ($2D\ge 4$ fixed) satisfies $\lambda_2\le 2\sqrt{2D-1}+o(1)$ a.a.s., giving a normalized-Laplacian spectral gap $\ge c_{\mathrm{exp}}(D)>0$. Such graphs also have vertex-connectivity equal to $2D$ a.a.s.\ (\cite[\S 2.6]{Wormald1999}), hence edge-connectivity $2D$. In particular, for all sufficiently large $m$ a simple $2D$-regular graph $\Gamma_v$ on $m^3$ vertices with both properties exists.

For (c): a connected $2D$-regular simple graph is bridgeless (a bridge in a connected $2k$-regular graph would create a component with odd degree sum, a contradiction). By Petersen's classical $2$-factorization theorem, every bridgeless regular graph of even degree decomposes into edge-disjoint $2$-factors. Applying this to $\Gamma_v$ yields $D$ edge-disjoint $2$-factors $F_1,\dots,F_D$; each $F_i$ is a union of cycles covering all $m^3$ vertices. Orient every cycle in $F_i$ to obtain a permutation $\sigma_{v,i}$ of $\{1,\dots,m^3\}$. Glue the $i$th outgoing face of block $t$ to the $i$th incoming face of block $\sigma_{v,i}(t)$. By construction, the resulting colored directed graph has exactly one incoming and one outgoing edge of each color at every vertex.

Let $h_0=h_0(D)>0$ be a uniform lower bound on the edge-expansion Cheeger constant of the chosen graphs $\Gamma_v$ (using the convention $h(\Gamma)=\min_{|S|\le n/2}|\partial_E S|/|S|$, where $\partial_E S$ denotes the set of edges between $S$ and its complement); by the discrete Cheeger inequality applied to the spectral gap $c_{\mathrm{exp}}(D)$, such a bound $h_0>0$ exists independently of~$m$.

To expose ports, we select one edge per color to serve as the external corridor attachment. For each color $i$, choose an edge $e_i$ of color $i$ (i.e.\ an edge $s_i\to\sigma_{v,i}(s_i)$) and delete it from the internal graph. The freed outgoing face of $s_i$ and the freed incoming face of $\sigma_{v,i}(s_i)$ become the port faces for color~$i$, to which the external corridor of that color will be attached. We choose the $D$ deleted edges to be pairwise vertex-disjoint by a greedy argument: each color class is a union of disjoint cycles on $m^3$ vertices with $m^3$ edges, so after choosing $i-1$ port edges (involving $2(i-1)$ vertices), each such vertex is incident to at most $2$ color-$i$ edges, giving at most $4(i-1)$ blocked color-$i$ edges; for $m^3>4D$ there remains a color-$i$ edge disjoint from the previous ones. This deletes exactly $D$ internal edges.

The modified internal graph $\Gamma_v^{\circ}$ has $2D$ vertices of degree $2D-1$ (the two endpoints of each deleted edge) and the remaining $m^3-2D$ vertices of degree~$2D$.

We claim $\Gamma_v^{\circ}$ retains a positive Cheeger constant uniformly in~$m$:
\begin{itemize}
\item \emph{Connectivity after port deletion.} The chosen realization has edge-connectivity~$2D$ (property~(ii) above). Since $2D>D$, deleting the $D$ port edges preserves connectivity.
\item \emph{Cheeger bound.} Deleting $D$ edges changes the edge boundary of any subset by at most $D$, hence for any $S$ with $|S|\le m^3/2$,
\[
|\partial S|_{\Gamma_v^{\circ}} \ge |\partial S|_{\mathrm{old}} - D \ge h_0|S| - D.
\]
If $|S|\ge 2D/h_0$ this gives $|\partial S|_{\Gamma_v^{\circ}}/|S| \ge h_0/2$. If $|S|<2D/h_0$ then connectivity implies $|\partial S|_{\Gamma_v^{\circ}}\ge 1$, giving $|\partial S|/|S|\ge h_0/(2D)$.
\end{itemize}
Thus $\Gamma_v^{\circ}$ has Cheeger constant $\ge \min(h_0/2,  h_0/(2D))=h_0/(2D)>0$, uniformly in~$m$. (This Cheeger bound, together with the uniform degree bound $\le 2D$, is the only property of $\Gamma_v^{\circ}$ used in the sequel; see Lemma~\ref{lem:cluster_poincare}.)

\begin{remark}[Probabilistic method]
\label{rem:probabilistic}
The expander wiring is constructed via the probabilistic method. Friedman's theorem \cite{Friedman2008} and the connectivity results of \cite[\S 2.6]{Wormald1999} guarantee the existence of a simple $2D$-regular graph with near-optimal spectral gap and full edge-connectivity on $m^3$ vertices, for each sufficiently large~$m$. Petersen's classical $2$-factorization theorem then decomposes this graph into $D$ directed permutation color classes, yielding the required Schreier coloring. The Cheeger analysis after port deletion is deterministic, depending only on the Cheeger constant and edge-connectivity of the chosen graph. The resulting construction is therefore existential (non-constructive) but yields a definite covering manifold $M_m$ for each~$m$.
\end{remark}

\paragraph{The global Schreier covering.}
Assemble all clusters and corridors (with the unused faces on corridor interior blocks sealed by self-loops in the other colors, exactly as in the definition of $\widetilde{S}_i$; concretely, for each such block the two faces $H_j^+$ and $H_j^-$ of color $j\neq c(e)$ are identified with each other, which simply reconstructs the smooth local geometry of~$B$ across that interface and introduces no singularities) to obtain a finite connected directed $D$-regular colored graph $\Gamma_m$ in which each vertex has exactly one incoming and one outgoing edge of each color. (Connectivity follows from the construction: the corridors connect clusters according to $K_N$, and the internal expander wiring makes each cluster connected.) Such a connected $D$-regular graph is precisely the Schreier graph of a transitive action of $F_D$ on a finite set, and hence of a finite-index subgroup $\Lambda_m<F_D$: each generator of $F_D$ acts as a permutation of the vertex set, and $\Lambda_m$ is the stabilizer of a chosen basepoint (see, e.g., \cite[\S 11.1]{HLW2006}). The subgroup $\Lambda_m$ determines a finite covering $X_m'\to X_D$ of the bouquet of $D$ circles; pulling this back along the surjection $f:B\to X_D$ (Lemma~\ref{lem:bouquet_map}) yields a finite covering
\[
M_m := f^*(X_m') \longrightarrow B.
\]
The pullback $M_m$ is connected: since $f_*:\pi_1(B)\twoheadrightarrow F_D$ is surjective and $\Lambda_m<F_D$ is a finite-index subgroup, the image $f_*(\pi_1(B))$ acts transitively on $F_D/\Lambda_m$, so the pullback covering has a single connected component (by the standard orbit description of covering components; see, e.g., \cite[Ch.~1, \S 1.3]{Hatcher}). Since $M_m$ is a genuine Riemannian covering of the smooth hyperbolic manifold~$B$, it inherits the constant curvature $\kappa=-1$ metric by pullback and is itself a smooth Riemannian manifold (in particular, all gluing interfaces are smooth). The decomposition of $M_m$ into blocks isometric to~$F$, with gluing pattern given by $\Gamma_m$, follows from Remark~\ref{rem:pullback_decomposition}: the triviality of the covering over~$F^\circ$ (ensured by Lemma~\ref{lem:bouquet_map}(iii), which places $f(F^\circ)$ in the contractible set $X_D\setminus\{p_1,\dots,p_D\}$) gives $|S|$ isometric copies of~$F^\circ$ in~$M_m$; cutting along the lifts of the $H_i$ produces $|S|$ copies of the compact block~$F$, glued along their boundary faces according to~$\Gamma_m$.

\subsection{Discrete scaling (graph model)}

Define a discrete Laplacian $L_m$ on $K_N$ with:
\[
\text{vertex measures }  \mathsf{V}_v(m) := m^3 V_F,
\qquad
\text{edge weights } \mathsf{w}_e(m):=\frac{C_{c(e)}}{K_e(m)}.
\]
This is the macroscopic ``heavy vertex'' graph model.

\begin{lemma}[Scaling of the macroscopic graph]
\label{lem:scaling}
For each fixed $k\in\{1,\dots,N-1\}$,
\[
\lambda_k(L_m) = \frac{\mu_k}{m^4 V_F} (1+O(m^{-1})),
\]
where $\mu_k=\lambda_k(L^*)$ are the eigenvalues of the reference Laplacian $L^*$ on $K_N$ with unit vertex measures and edge weights $w_e^*$. In particular, for $k\le n$ (where $\mu_k=\lambda_k^*V_F$ by construction) this gives $\lambda_k(L_m) = m^{-4}\lambda_k^* (1+O(m^{-1}))$, and for $n+1\le k\le N-1$ (where $\mu_k>(\lambda_n^*+1)V_F$ by the padding choice in~\S4.3) it gives $m^4\lambda_k(L_m)\to \mu_k/V_F > \lambda_n^*+1$.
\end{lemma}

\begin{proof}
By definition of $K_e(m)$,
\[
\mathsf{w}_e(m)=\frac{C_{c(e)}}{\lfloor m C_{c(e)}/w_e^*\rfloor}
= \frac{w_e^*}{m} \bigl(1+O(m^{-1})\bigr),
\]
uniformly over the finitely many edges of $K_N$.

Let $L^*$ denote the Laplacian on $K_N$ with unit vertex measures and edge weights $w_e^*$. Then $L_m$ is a perturbation of the scaled operator
\[
\frac{1}{m^4 V_F} L^*,
\]
because multiplying all edge weights by $1/m$ multiplies the Laplacian by $1/m$, and multiplying all vertex measures by $m^3 V_F$ divides it by $m^3 V_F$ (since the vertex measure appears in the denominator of~\eqref{eq:discrete_laplacian}).

Since the coefficient errors are $O(m^{-5})$ at the operator level and $K_N$ is finite-dimensional, eigenvalues perturb by $O(m^{-5})$ in absolute size, hence by $O(m^{-1})$ in relative size after multiplying by $m^4$. Therefore, for every $k\in\{1,\dots,N-1\}$,
\[
\lambda_k(L_m)=\frac{1}{m^4V_F}\mu_k (1+O(m^{-1})).
\]
For $k\le n$, $\mu_k=\lambda_k^*V_F$ by construction, giving $\lambda_k(L_m)= m^{-4}\lambda_k^* (1+O(m^{-1}))$.
\end{proof}

\subsection{Spectral reduction: from \texorpdfstring{$M_m$}{Mm} to \texorpdfstring{$L_m$}{Lm}}

Let $\nu_k(M_m)$ denote the $k$th eigenvalue (counting multiplicity) of the Laplace-Beltrami operator on $(M_m,g_{M_m})$ (curvature $-1$).

We work throughout with $L^2$-normalized functions: $\|u\|_{L^2(M_m)}=1$.
We write $\langle \bar u, L_m \bar u\rangle_{\mathsf{V}}:=\sum_v \mathsf{V}_v \bar u_v (L_m\bar u)_v = \sum_e \mathsf{w}_e(\bar u_{v'}-\bar u_v)^2$ for the Dirichlet energy associated to the weighted graph Laplacian $L_m$ (note: this is the $\ell^2(\mathsf{V})$-inner product, \emph{not} the Euclidean product $\bar u^{\mathsf T}L_m\bar u$, which would include spurious factors of $\mathsf{V}_v^{-1}$).
Fix the decomposition of $M_m$ into the $N$ vertex clusters $U_{v,m}$
($\Vol(U_{v,m})=m^3 V_F$) and the $\binom{N}{2}$ corridors $W_{e,m}$
($\Vol(W_{e,m})=K_e(m)V_F = O(m)$), sharing only their common port faces.

The proof of the spectral reduction proceeds through six modular ingredients, stated below as standalone results and then assembled into the main comparison:

\begin{enumerate}
\item[\textbf{(A)}] \emph{Uniform cluster Poincare inequality} (Lemma~\ref{lem:cluster_poincare}): a uniform Neumann gap on expander-glued blocks.
\item[\textbf{(B)}] \emph{Quantitative trace estimate at ports} (estimate~\eqref{eq:trace_from_cluster}): flux averages at port faces approximate cluster means.
\item[\textbf{(C)}] \emph{Corridor effective conductance} (Lemma~\ref{lem:corridor}): upper and lower bounds on corridor Dirichlet energies in terms of endpoint boundary data.
\item[\textbf{(C$'$)}] \emph{Corridor $L^2$-mass bound}: low-energy functions have negligible $L^2$ mass on corridors.
\item[\textbf{(D)}] \emph{Upper comparison via prolongation}: $\nu_k(M_m)\le\lambda_k(L_m)(1+O(m^{-1}))$.
\item[\textbf{(E)}] \emph{Lower comparison via cluster averaging}: $\nu_k(M_m)\ge\lambda_k(L_m)(1-o(1))$.
\end{enumerate}
These are combined with a parasitic eigenvalue bound (Proposition~\ref{prop:parasitic}) that ensures no unwanted eigenvalues enter the target spectral window.

\smallskip
\noindent\textbf{(A) Uniform Poincare inequality on clusters.}

\begin{lemma}[Uniform Neumann gap for expander-glued blocks]
\label{lem:cluster_poincare}
Let $F$ be a compact Riemannian manifold with boundary having $2D$ boundary
faces, and let $\Gamma$ be a connected graph on $m^3$ vertices with
combinatorial Cheeger constant $h(\Gamma)\ge h_0>0$.
Let $U$ be the Riemannian manifold obtained by gluing $m^3$ copies of $F$ along
their faces according to $\Gamma$ (with Neumann boundary on the $O(D)$ exposed
port faces). Then there exists $C_P=C_P(F,h_0)>0$, independent of $m$,
such that for every $u\in H^1(U)$,
\begin{equation}
\label{eq:cluster_poincare}
\int_{U} |u-\bar u|^2 \le C_P \int_{U} |\nabla u|^2,
\qquad
\bar u := \frac{1}{\Vol(U)}\int_{U} u.
\end{equation}
\end{lemma}

\begin{proof}
We give a self-contained domain-decomposition argument (compare Kanai \cite{Kanai1985} and Mantuano \cite{Mantuano2005} for the related rough-isometry discretization approach on closed manifolds).
Let $F_1,\dots,F_{m^3}$ denote the copies of~$F$ inside~$U$, and
let $\bar u_j:=V_F^{-1}\int_{F_j}u$ be the block mean.
For any $u\in H^1(U)$ with global mean $\bar u=0$
(equivalently $\sum_j\bar u_j=0$), we estimate
$\int_U|u|^2$ from above by combining local and global Poincare inequalities.

\smallskip
\emph{Step~1 (local Poincare on blocks).}
Each block $F_j$ is isometric to the fixed compact manifold~$F$ with smooth boundary, so
$\lambda_1^{\mathrm{Neu}}(F)=:\lambda_F>0$ gives
\begin{equation}
\label{eq:local_poincare}
\int_{F_j}|u-\bar u_j|^2\le \frac{1}{\lambda_F}\int_{F_j}|\nabla u|^2.
\end{equation}
Summing over $j$: $\sum_j\int_{F_j}|u-\bar u_j|^2\le (1/\lambda_F)\int_U|\nabla u|^2$.

\smallskip
\emph{Step~2 (coupling adjacent block means to gradient).}
For blocks $F_j,F_k$ glued along a shared face $\Sigma_{jk}$, let
$\bar u_{\Sigma}:=|\Sigma_{jk}|^{-1}\int_{\Sigma_{jk}}u dA$ be the face average.
The trace inequality on the fixed block~$F$ gives
$|\bar u_j-\bar u_\Sigma|\le C_{\mathrm{adj}}\|\nabla u\|_{L^2(F_j)}$
for a constant $C_{\mathrm{adj}}$ depending only on~$F$ (via trace + Cauchy-Schwarz + Poincare), and similarly on $F_k$.
By the triangle inequality:
\begin{equation}
\label{eq:adj_means}
(\bar u_j-\bar u_k)^2\le 4C_{\mathrm{adj}}^2\Bigl(\int_{F_j}|\nabla u|^2+\int_{F_k}|\nabla u|^2\Bigr).
\end{equation}

\smallskip
\emph{Step~3 (discrete Poincare from Cheeger constant).}
The gluing graph $\Gamma$ has $h(\Gamma)\ge h_0>0$ by hypothesis and uniformly bounded degree $d_{\max}\le 2D$. By Mohar's discrete Cheeger inequalities \cite{Mohar1989} for the combinatorial Laplacian, the isoperimetric number $i(\Gamma)$ and the first nonzero eigenvalue $\lambda_1(\Gamma)$ satisfy $i(\Gamma)\le\sqrt{\lambda_1(2d_{\max}-\lambda_1)}$, which gives $\lambda_1(\Gamma)\ge i(\Gamma)^2/(2d_{\max})$ (see also \cite[\S 4.4--4.5]{HLW2006} for the normalized-Laplacian version on regular graphs). Since $i(\Gamma)\ge h(\Gamma)\ge h_0$, we obtain $\lambda_1(\Gamma)\ge h_0^2/(4D)=:c_{\mathrm{gap}}(h_0,D)>0$.
Since $\sum_j\bar u_j=0$, the discrete Poincare inequality gives
$\sum_j \bar u_j^2 \le (1/c_{\mathrm{gap}})\sum_{(j,k)\in E}(\bar u_j-\bar u_k)^2$.
Substituting \eqref{eq:adj_means} and noting each block appears in at most $2D$ adjacency terms:
\begin{equation}
\label{eq:discrete_poincare}
V_F\sum_j \bar u_j^2 \le \frac{8D V_F C_{\mathrm{adj}}^2}{c_{\mathrm{gap}}}\int_U|\nabla u|^2.
\end{equation}

\smallskip
\emph{Step~4 (combining).}
Since $\int_{F_j}(u-\bar u_j)=0$, the cross term vanishes in $\int_{F_j}|u|^2=\int_{F_j}|u-\bar u_j|^2+V_F\bar u_j^2$. Hence:
\[
\int_U|u|^2=\sum_j\int_{F_j}|u-\bar u_j|^2 + V_F\sum_j\bar u_j^2
\le\underbrace{\Bigl(\frac{1}{\lambda_F}+\frac{8D V_F C_{\mathrm{adj}}^2}{c_{\mathrm{gap}}}\Bigr)}_{=: C_P}\int_U|\nabla u|^2.
\]
All constants ($\lambda_F,D,V_F,C_{\mathrm{adj}},c_{\mathrm{gap}}$) depend only on $F$ and~$D$, not on~$m$.
\end{proof}

We write $\bar u_v$ for the mean of $u$ over the cluster $U_{v,m}$.

\smallskip
\noindent\textbf{(B) Quantitative trace estimate at ports.}
Fix a port face $\Sigma$ bounding a cluster $U_{v,m}$, and let $F_\Sigma$ be
the block of the cluster adjacent to that face. Since $F$ has fixed geometry,
the standard trace inequality on $F$ yields a constant $C_{\mathrm{tr}}>0$ such that
\[
\|u-\bar u_v\|_{L^2(\Sigma)}^2
\le C_{\mathrm{tr}}\Bigl(\|u-\bar u_v\|_{L^2(F_\Sigma)}^2+\|\nabla u\|_{L^2(F_\Sigma)}^2\Bigr).
\]
Since $F_\Sigma\subset U_{v,m}$, we have
$\|u-\bar u_v\|_{L^2(F_\Sigma)}^2 \le \|u-\bar u_v\|_{L^2(U_{v,m})}^2
\le C_P \int_{U_{v,m}} |\nabla u|^2$
by \eqref{eq:cluster_poincare}, and similarly
$\|\nabla u\|_{L^2(F_\Sigma)}^2 \le \int_{U_{v,m}} |\nabla u|^2$.
Therefore
\begin{equation}
\label{eq:trace_from_cluster}
\|u-\bar u_v\|_{L^2(\Sigma)}^2 \le C' \int_{U_{v,m}} |\nabla u|^2,
\end{equation}
with $C'=C_{\mathrm{tr}}(C_P+1)$, independent of $m$.

\smallskip
\noindent\textbf{(C) Replacing flux averages by cluster means.}
For a corridor $W_{e,m}$ of color $i=c(e)$ connecting clusters $v$ and $v'$,
Lemma~\ref{lem:corridor}(i) gives
\[
\int_{W_{e,m}} |\nabla u|^2
\ge \frac{C_i}{K_e} \bigl(\langle u\rangle_+ - \langle u\rangle_-\bigr)^2,
\]
where $\langle u\rangle_\pm$ are the normalized flux averages at the two ends.
We now replace these by the cluster means $\bar u_v, \bar u_{v'}$.

By definition,
$\langle u\rangle_+ = \frac{1}{C_i}\int_{\Sigma^+} u \partial_n\chi_i dA$,
so
\[
|\langle u\rangle_+ - \bar u_{v'}|
= \frac{1}{C_i}\left|\int_{\Sigma^+} (u-\bar u_{v'}) \partial_n\chi_i dA\right|
\le \frac{\|\partial_n\chi_i\|_{L^2(\Sigma^+)}}{C_i} \|u-\bar u_{v'}\|_{L^2(\Sigma^+)}.
\]
Since $\partial_n\chi_i$ is smooth on the fixed face geometry,
$\|\partial_n\chi_i\|_{L^2(\Sigma^+)}/C_i$ is a fixed constant $C_\chi$
depending only on the block~$F$. (Since $\chi_i$ is a fixed smooth function on the compact block~$F$, the quantity $\|\partial_n\chi_i\|_{L^2(\Sigma^\pm)}$ is a definite positive constant determined by the Riemannian geometry of~$F$; we simply define $C_\chi:=\|\partial_n\chi_i\|_{L^2(\Sigma^\pm)}/C_i$.)
Combined with \eqref{eq:trace_from_cluster}:
\begin{equation}
\label{eq:flux_vs_mean}
|\langle u\rangle_\pm - \bar u_w|
\le C_\chi\sqrt{C'}\left(\int_{U_{w,m}} |\nabla u|^2\right)^{1/2}
=: A\left(\int_{U_{w,m}} |\nabla u|^2\right)^{1/2},
\end{equation}
where $w$ is the cluster adjacent to that end and $A=C_\chi\sqrt{C'}$ is
independent of~$m$.

Now write $\langle u\rangle_+ - \langle u\rangle_-
= (\bar u_{v'} - \bar u_v) + \eta_e$, where
\[
|\eta_e|
\le A\left(\int_{U_{v,m}}|\nabla u|^2\right)^{1/2}
+ A\left(\int_{U_{v',m}}|\nabla u|^2\right)^{1/2}.
\]
Then
\begin{align}
\frac{C_i}{K_e}\bigl(\langle u\rangle_+ - \langle u\rangle_-\bigr)^2
&= \frac{C_i}{K_e}\bigl[(\bar u_{v'}-\bar u_v)^2 + 2(\bar u_{v'}-\bar u_v)\eta_e + \eta_e^2\bigr] \notag\\
&\ge \frac{C_i}{K_e}(\bar u_{v'}-\bar u_v)^2
- \frac{C_i}{K_e}\bigl[2|\bar u_{v'}-\bar u_v||\eta_e| + \eta_e^2\bigr].
\label{eq:corridor_with_correction}
\end{align}
We estimate the correction terms under the assumption that $u$ has
$\|u\|_{L^2}=1$ and lies in the low-energy window, i.e.,
$\int_{M_m}|\nabla u|^2 = \nu = O(m^{-4})$.

Since $\Vol(U_{v,m})=m^3 V_F$ and $\|u\|_{L^2}=1$, the cluster means satisfy
$\bar u_v = O(m^{-3/2})$, hence $|\bar u_{v'}-\bar u_v|=O(m^{-3/2})$.
The error term satisfies $|\eta_e|\le 2A\nu^{1/2}=O(m^{-2})$.
Therefore
\[
\frac{C_i}{K_e}\bigl[2|\bar u_{v'}-\bar u_v||\eta_e|+\eta_e^2\bigr]
= O(m^{-1})\bigl[O(m^{-3/2}\cdot m^{-2}) + O(m^{-4})\bigr]
= O(m^{-9/2}).
\]
Summing over the finitely many edges of $K_N$:
\begin{equation}
\label{eq:corridor_sum_lower}
\sum_e \int_{W_{e,m}} |\nabla u|^2
\ge
\sum_e \mathsf{w}_e(\bar u_{v'}-\bar u_v)^2 - O(m^{-9/2})
=
\langle \bar u, L_m \bar u\rangle_{\mathsf{V}} - O(m^{-9/2}).
\end{equation}

\smallskip
\noindent\textbf{(C${}'$) Corridor $L^2$ mass bound.}

\begin{lemma}[Corridor $L^2$-mass estimate]
\label{lem:corridor_mass}
For any $L^2$-normalized $u\in H^1(M_m)$ with $\int_{M_m}|\nabla u|^2=\nu$,
\begin{equation}
\label{eq:corridor_mass}
\sum_e \int_{W_{e,m}} |u|^2
\le
O(m^{-2}) + O(m^2) \nu.
\end{equation}
In particular, for low-energy functions with $\nu=O(m^{-4})$, the corridor $L^2$ mass is $O(m^{-2})$. If additionally all cluster means vanish ($\bar u_v=0$ for all $v$), then
\begin{equation}
\label{eq:corridor_mass_general}
\sum_e \int_{W_{e,m}} |u|^2
\le
O(m^2) \nu.
\end{equation}
\end{lemma}

\begin{proof}
We bound the $L^2$ mass of $u$ on corridors using the energy.
Consider a single corridor $W_{e,m}$ with $K_e=K_e(m)$ blocks
$F_1,\dots,F_{K_e}$, connecting clusters $v$ and $v'$.

On each block $F_j$, let $\bar u_j := \frac{1}{V_F}\int_{F_j} u$ be the
block mean. Consecutive block means satisfy
$|\bar u_{j+1}-\bar u_j|\le C_{\mathrm{adj}} \|\nabla u\|_{L^2(F_j\cup F_{j+1})}$
for a constant $C_{\mathrm{adj}}$ depending only on~$F$:
indeed, if $\Sigma_{j,j+1}=F_j\cap F_{j+1}$ is the shared face with average
$\bar u_{\Sigma_{j,j+1}}:=|\Sigma|^{-1}\int_{\Sigma_{j,j+1}}u dA$,
then by Poincare and the trace inequality on~$F$,
$|\bar u_j-\bar u_{\Sigma_{j,j+1}}|\le C_{\mathrm{tr}}'\|\nabla u\|_{L^2(F_j)}$,
and likewise on $F_{j+1}$; the triangle inequality gives
$C_{\mathrm{adj}}=2C_{\mathrm{tr}}'$.
The boundary block mean $\bar u_1$ is controlled by the adjacent cluster via the shared port face $\Sigma$. Note that $F_1\subset W_{e,m}$ and the cluster block $F_\Sigma\subset U_{v,m}$ share the interface~$\Sigma$ but are otherwise disjoint. We use the triangle inequality:
\[
|\bar u_1 - \bar u_v|
\le
\underbrace{|\bar u_1 - \bar u_\Sigma|}_{\text{corridor side}}
+
\underbrace{|\bar u_\Sigma - \bar u_v|}_{\text{cluster side}},
\]
where $\bar u_\Sigma := \frac{1}{|\Sigma|}\int_\Sigma u dA$ is the face average. For the corridor side, the standard trace/Poincare inequality on the fixed block $F_1$ gives
$|\bar u_1 - \bar u_\Sigma| \le C_{\mathrm{tr}}' \|\nabla u\|_{L^2(F_1)}$.
For the cluster side, $\bar u_\Sigma - \bar u_v = \frac{1}{|\Sigma|}\int_\Sigma(u-\bar u_v) dA$, so by Cauchy-Schwarz and \eqref{eq:trace_from_cluster}:
$|\bar u_\Sigma - \bar u_v| \le |\Sigma|^{-1/2}\sqrt{C'} \|\nabla u\|_{L^2(U_{v,m})}$.
Combining: $|\bar u_1 - \bar u_v| \le C_1\bigl(\|\nabla u\|_{L^2(F_1)} + \|\nabla u\|_{L^2(U_{v,m})}\bigr)$ for a constant $C_1$ depending only on $F$. Hence $|\bar u_1|\le |\bar u_v| + C_1\bigl(\|\nabla u\|_{L^2(W_{e,m})} + \|\nabla u\|_{L^2(U_{v,m})}\bigr)$.
By Cauchy-Schwarz applied to the telescoping sum (noting that the overlapping domains $F_k\cup F_{k+1}$ count each interior block at most twice, contributing a factor of $\sqrt{2}$ absorbed into $C_{\mathrm{adj}}$),
\[
|\bar u_j| \le |\bar u_1| + C_{\mathrm{adj}}\sqrt{j} \|\nabla u\|_{L^2(W_{e,m})},
\]
so $\bar u_j^2 \le 2\bar u_1^2 + 2C_{\mathrm{adj}}^2 j \|\nabla u\|_{L^2(W_{e,m})}^2$.
Summing over blocks and using the single-block Poincare inequality
$\int_{F_j}|u-\bar u_j|^2\le C_P^{(1)}\int_{F_j}|\nabla u|^2$:
\begin{align*}
\int_{W_{e,m}} |u|^2
&= \sum_{j=1}^{K_e}\Bigl[V_F\bar u_j^2 + \int_{F_j}|u-\bar u_j|^2\Bigr]\\
&\le 2K_e V_F\bar u_1^2 + 2C_{\mathrm{adj}}^2 K_e^2 V_F\|\nabla u\|_{L^2(W_e)}^2
+ C_P^{(1)}\int_{W_e}|\nabla u|^2.
\end{align*}
For general $u$, the cluster mean satisfies $\bar u_v^2\le(\Vol U_{v,m})^{-1}\int_{U_v}|u|^2= O(m^{-3})$ (using only $\|u\|_{L^2}=1$ and $\Vol(U_{v,m})=m^3V_F$). Hence
\[
\bar u_1^2\le 2\bar u_v^2+2C_1^2\bigl(\|\nabla u\|_{L^2(F_1)}^2+\|\nabla u\|_{L^2(U_v)}^2\bigr)
= O(m^{-3})+O\bigl(\|\nabla u\|_{L^2(F_1)}^2+\|\nabla u\|_{L^2(U_v)}^2\bigr).
\]
Since $K_e=O(m)$, the first displayed bound gives
\[
\int_{W_{e,m}} |u|^2
\le
O(m^{-2}) + O(m)\|\nabla u\|_{L^2(U_v)}^2
+ O(m^2)\int_{W_e}|\nabla u|^2.
\]
(Here $v$ is the cluster at the left end of the corridor; a symmetric bound holds with $U_{v'}$ in place of $U_v$.)
Summing over all corridors: for each corridor we use the bound at one chosen endpoint, and each cluster $U_v$ is an endpoint of at most $N-1$ corridors (one per macroscopic edge of $K_N$ at~$v$), so the $O(m)$-terms sum to $O(m)\nu$, which is absorbed by $O(m^2)\nu$. This establishes~\eqref{eq:corridor_mass}.
For low-energy functions with $\int|\nabla u|^2=\nu=O(m^{-4})$, both terms are
$O(m^{-2})$, confirming that corridor $L^2$ mass is negligible.

More generally, \eqref{eq:corridor_mass} shows that for \emph{any}
$L^2$-normalized $u$ with all cluster means equal to zero,
\eqref{eq:corridor_mass_general} holds,
since the $O(m^{-2})$ constant term is absorbed when $\bar u_v=0$ forces
$\bar u_1^2\le C_1^2(\|\nabla u\|_{L^2(U_v)}+\|\nabla u\|_{L^2(W_e)})^2$ (without the $|\bar u_v|$ term).
\end{proof}

\smallskip
\noindent\textbf{(D) Upper bound via prolongation.}
Given a discrete vector $x=(x_v)_{v\in V(K_N)}$, define $u_x\in H^1(M_m)$ by:
\begin{itemize}
\item $u_x\equiv x_v$ on each cluster $U_{v,m}$;
\item on each corridor $W_{e,m}$ connecting $v$ to $v'$ of color $i=c(e)$, let
$u_x$ be the \emph{energy minimizer} with constant boundary data $u_x\equiv x_v$
on the entrance face and $u_x\equiv x_{v'}$ on the exit face.
\end{itemize}
By Lemma~\ref{lem:corridor}(ii), the energy of this minimizer satisfies
$\int_{W_e}|\nabla u_x|^2 \le [C_i/K_e + C_{\mathrm{bl}}/K_e^2](x_v-x_{v'})^2$.
This produces a globally $H^1$ function: trace matching across each port face ensures continuity, hence $u_x\in H^1(M_m)$ and the Rayleigh quotient is well-defined.
Since $u_x$ is constant on clusters, $\nabla u_x=0$ there, so
\[
\int_{M_m}|\nabla u_x|^2
= \sum_e \int_{W_{e,m}}|\nabla u_x|^2
\le \sum_e \left[\frac{C_i}{K_e}+\frac{C_{\mathrm{bl}}}{K_e^2}\right](x_v-x_{v'})^2
= \sum_e \mathsf{w}_e(x_v-x_{v'})^2 (1+O(m^{-1})).
\]
For the $L^2$ norm, the cluster contribution is $\sum_v m^3 V_F  x_v^2$.
On each corridor, the minimizer $u_x$ satisfies $\min(x_v,x_{v'})\le u_x\le\max(x_v,x_{v'})$ by the maximum principle (it is harmonic on each block interior with boundary data bounded by the endpoint values), hence
$\int_{W_e}|u_x|^2 \le K_e V_F\cdot\max(x_v^2,x_{v'}^2)$.
With the discrete normalization $\sum_v \mathsf{V}_v x_v^2=1$
(so $x_v=O(m^{-3/2})$), corridor volume $O(m)$ gives
\[
\sum_e \int_{W_e} |u_x|^2 = O(m\cdot m^{-3}) = O(m^{-2}).
\]
Therefore
\[
\frac{\int|\nabla u_x|^2}{\int|u_x|^2}
= \frac{\sum_e \mathsf{w}_e(x_v-x_{v'})^2 (1+O(m^{-1}))}
       {\sum_v \mathsf{V}_v x_v^2 (1+O(m^{-2}))}
= R_{\mathrm{disc}}(x) (1+O(m^{-1})),
\]
where $R_{\mathrm{disc}}$ denotes the Rayleigh quotient for $L_m$.
Since the Dirichlet problem on each corridor is linear in its boundary data, the prolongation map $x\mapsto u_x$ is linear; it is injective because $u_x\equiv x_v$ on each cluster. The image of any $(k+1)$-dimensional subspace of $\mathbb{R}^N$ is therefore a $(k+1)$-dimensional subspace of $H^1(M_m)$.
By the min-max characterization, choosing $x$ to run over the first $(k+1)$
discrete eigenspaces gives
\begin{equation}
\label{eq:upper_bound}
\nu_k(M_m)\le \lambda_k(L_m) (1+O(m^{-1})).
\end{equation}

\smallskip
\noindent\textbf{(E) Lower bound via cluster averaging.}
We prove $\nu_k(M_m)\ge \lambda_k(L_m) (1+o(1))$ for each fixed $k\le N-1$
using a min-max/dimension argument. (The argument is index-independent; the case $k\le n$ is the one needed for the target eigenvalues, but the same bound for $n+1\le k\le N-1$ is used in Theorem~\ref{thm:reduction} for the padding eigenvalues.)

Let $V\subset H^1(M_m)$ be any $(k+1)$-dimensional subspace. Define the
\emph{cluster averaging map} $\mathcal{A}:H^1(M_m)\to \mathbb{R}^N$ by
$\mathcal{A}(u)=(\bar u_v)_{v\in V(K_N)}$.

\emph{Case 1:} $\dim\mathcal{A}(V)\ge k+1$.
Then $\mathcal{A}(V)$ is a $(k+1)$-dimensional subspace of
$\mathbb{R}^N$.
By the discrete min-max characterization of $\lambda_k(L_m)$, there exists
$\bar u\in\mathcal{A}(V)\setminus\{0\}$ with
$R_{\mathrm{disc}}(\bar u)\ge \lambda_k(L_m)$.
Pick any nonzero $u_0\in V$ with $\mathcal{A}(u_0)=\bar u$, and set $u:=u_0/\|u_0\|_{L^2}$. Then $u\in V$ (since $V$ is a subspace), $\|u\|_{L^2}=1$, and $R(u)=R(u_0)$ by $0$-homogeneity of the Rayleigh quotient. Henceforth we write $\bar u_v:=\mathcal{A}(u)_v$ for the cluster means of the normalized~$u$; since $\mathcal{A}(u)=\mathcal{A}(u_0)/\|u_0\|_{L^2}$, $0$-homogeneity gives $R_{\mathrm{disc}}(\bar u)=R_{\mathrm{disc}}(\mathcal{A}(u_0))\ge\lambda_k(L_m)$.

If $R(u)\ge 2\lambda_k(L_m)$, then $\max_{w\in V,\|w\|=1}R(w)\ge 2\lambda_k(L_m) \ge \lambda_k(L_m)(1-o(1))$ and the desired lower bound holds trivially. Otherwise $R(u) \le 2\lambda_k(L_m) = O(m^{-4})$, which justifies the low-energy estimates in parts (C) and~(C$'$).

We now compare $R(u)$ with $R_{\mathrm{disc}}(\bar u)$.
From \eqref{eq:corridor_sum_lower},
\[
R(u) = \int_{M_m}|\nabla u|^2
\ge \sum_e \int_{W_e}|\nabla u|^2
\ge \langle \bar u, L_m \bar u\rangle_{\mathsf{V}} - O(m^{-9/2}).
\]
Since $\langle \bar u, L_m \bar u\rangle_{\mathsf{V}} = R_{\mathrm{disc}}(\bar u)\cdot\sum_v \mathsf{V}_v\bar u_v^2$, the $L^2$ decomposition gives:
\begin{align*}
1 = \int_{M_m}|u|^2
&= \sum_v\left[\Vol(U_v)\bar u_v^2 + \int_{U_v}|u-\bar u_v|^2\right]
+ \sum_e \int_{W_e}|u|^2 \\
&= \sum_v \mathsf{V}_v\bar u_v^2
+ \underbrace{\sum_v\int_{U_v}|u-\bar u_v|^2}_{\le  C_P\nu \text{by \eqref{eq:cluster_poincare}}}
+ \underbrace{\sum_e\int_{W_e}|u|^2}_{O(m^{-2}) \text{by \eqref{eq:corridor_mass}}},
\end{align*}
hence
\begin{equation}
\label{eq:L2_decomp}
\sum_v \mathsf{V}_v\bar u_v^2 = 1 - O(m^{-2}).
\end{equation}
Therefore
\[
R(u)
\ge
R_{\mathrm{disc}}(\bar u)\bigl(\sum_v \mathsf{V}_v\bar u_v^2\bigr) - O(m^{-9/2})
=
R_{\mathrm{disc}}(\bar u) (1-O(m^{-2})) - O(m^{-9/2})
\ge
\lambda_k(L_m) (1-o(1)),
\]
where in the last step we used $R_{\mathrm{disc}}(\bar u)\ge\lambda_k(L_m)$ together with
$\lambda_k(L_m)\asymp m^{-4}$ (by Lemma~\ref{lem:scaling}, since $\mu_k>0$ for $k\ge 1$),
so that $O(m^{-9/2})=o(m^{-4})=o(\lambda_k(L_m))$.

\emph{Case 2:} $\dim\mathcal{A}(V)\le k$.
By rank-nullity (since $\dim V = k+1 > k \ge \dim\mathcal{A}(V)$), $\ker\mathcal{A}\cap V$ has dimension $\ge 1$. Choose $u\in\ker\mathcal{A}\cap V$
with $\|u\|=1$. Since $\bar u_v=0$ for all $v$, Proposition~\ref{prop:parasitic}
below gives $R(u)\ge c_0/m^2$. Since $\lambda_k(L_m)=O(m^{-4})$, we have
$R(u)\gg \lambda_k(L_m)$ for large $m$.

In both cases, $\max_{u\in V,\|u\|=1}R(u)\ge \lambda_k(L_m)(1-o(1))$.
Since $V$ was an arbitrary $(k+1)$-dimensional subspace, the min-max
characterization gives
\begin{equation}
\label{eq:lower_bound}
\nu_k(M_m)\ge \lambda_k(L_m) (1-o(1)).
\end{equation}

\begin{remark}
\label{rem:error_rate}
The error rate $O(m^{-1})$ in the upper bound~\eqref{eq:upper_bound} is sharper than the $O(m^{-1/2})$ implicit in the lower bound~\eqref{eq:lower_bound}. The bottleneck in the lower bound is the crude global bound $\int_{U_{v,m}}|\nabla u|^2\le \nu=O(m^{-4})$ used to estimate $|\eta_e|=O(m^{-2})$, combined with $|\bar u_{v'}-\bar u_v|=O(m^{-3/2})$ (from $L^2$-normalization and cluster volume $m^3 V_F$); the resulting cross-term is $O(m^{-9/2})$, giving relative error $O(m^{-1/2})$ after dividing by $\lambda_k(L_m)\asymp m^{-4}$ (Lemma~\ref{lem:scaling}). We note that the cluster-mean scaling $\bar u_v=\Theta(m^{-3/2})$ is sharp and \emph{cannot} be improved: the eigenvector components have $\Theta(1)$ relative variation, so $|\bar u_{v'}-\bar u_v|=\Theta(m^{-3/2})$.

The actual looseness lies in the estimate of cluster-internal energy: for an eigenfunction at eigenvalue $\nu_k=O(m^{-4})$, almost all gradient energy is dissipated in the $O(m)$-length corridors rather than in the $O(1)$-resistance expander clusters, suggesting $\int_{U_{v,m}}|\nabla u|^2=O(m^{-5})$ rather than $O(m^{-4})$. Substituting this into the $\eta_e$ bound would improve the relative error to $O(m^{-1})$, matching the upper bound. Making this rigorous would require a quantitative energy-partition lemma that we do not pursue here.

We note that the $O(m^{-1})$ upper-bound rate is likely sharp: the floor function in $K_e(m)=\lfloor m C_{c(e)}/w_e^*\rfloor$ introduces an $O(m^{-1})$ discretization error that propagates to the eigenvalue approximation.
\end{remark}

We now state the parasitic eigenvalue bound, which is needed in both Case~2 of part~(E) and in the proof of the main spectral reduction theorem.

\begin{proposition}[Parasitic eigenvalue bound]
\label{prop:parasitic}
There exists $c_0=c_0(F,D,w^*)>0$ such that for all sufficiently large~$m$,
every $L^2$-normalized function $u$ on $M_m$ with $\bar u_v=0$ for all
$v\in V(K_N)$ satisfies
\[
\int_{M_m}|\nabla u|^2 \ge \frac{c_0}{m^2}.
\]
In particular, $\nu_N(M_m)\ge c_0/m^2$.
\end{proposition}

\begin{proof}
Let $u\in H^1(M_m)$ with $\|u\|_{L^2}=1$ and $\bar u_v=0$ for all $v$.
Set $\nu:=\int_{M_m}|\nabla u|^2$.

By the cluster Poincare inequality \eqref{eq:cluster_poincare} with
$\bar u_v=0$:
\begin{equation}
\label{eq:cluster_mass_parasitic}
\sum_v \int_{U_{v,m}} |u|^2
= \sum_v \int_{U_{v,m}}|u-\bar u_v|^2
\le C_P \nu.
\end{equation}
By the corridor mass bound \eqref{eq:corridor_mass_general}
(valid when all cluster means vanish):
\begin{equation}
\label{eq:corridor_mass_parasitic}
\sum_e \int_{W_{e,m}}|u|^2
\le C_{\mathrm{corr}} m^2 \nu,
\end{equation}
for a constant $C_{\mathrm{corr}}$ depending only on $F$ and $D$ (and, through $K_e(m)=O(m)$, on the prescribed edge weights $w^*$; since $w^*$ is fixed before the limit $m\to\infty$, this dependence is harmless).
Since $\|u\|_{L^2}^2=1$, combining gives
\[
1 \le C_P \nu + C_{\mathrm{corr}} m^2 \nu
= (C_P + C_{\mathrm{corr}} m^2) \nu.
\]
Hence $\nu\ge 1/(C_P + C_{\mathrm{corr}} m^2)\ge c_0/m^2$ for
$c_0=1/(2C_{\mathrm{corr}})$ and all $m$ large enough that
$C_{\mathrm{corr}} m^2\ge C_P$.

For the eigenvalue statement: any $(N+1)$-dimensional subspace
$V\subset H^1(M_m)$ satisfies $\dim(\ker\mathcal{A}\cap V)\ge 1$
(since $\mathcal{A}:H^1\to\mathbb{R}^N$), so $V$ contains an
$L^2$-normalized function with all cluster means zero, whose Rayleigh
quotient is $\ge c_0/m^2$. By the min-max characterization,
$\nu_N(M_m)\ge c_0/m^2$.
\end{proof}

We can now state the main spectral reduction result.

\begin{theorem}[Spectral reduction to $K_N$ in the heavy-vertex regime]
\label{thm:reduction}
Fix the target list $0=\lambda_0^*<\cdots<\lambda_n^*$ and the associated construction above. Then $\nu_0(M_m)=\lambda_0(L_m)=0$, and for each $k\in\{1,\dots,N-1\}$,
\[
\nu_k(M_m) = \lambda_k(L_m) (1+o(1))\qquad (m\to\infty).
\]
In particular, $m^4\nu_k(M_m)\to \lambda_k^*$ for $k\le n$, and $m^4\nu_k(M_m)\to\mu_k/V_F>\lambda_n^*+1$ for $n+1\le k\le N-1$ (where $\mu_{n+1},\dots,\mu_{N-1}$ are the padding eigenvalues chosen strictly above $(\lambda_n^*+1) V_F$ in~\S4.3).
Moreover, $\nu_N(M_m)\ge c_0/m^2$ for a constant $c_0>0$ independent of~$m$ (Proposition~\ref{prop:parasitic}), so that $m^4\nu_N(M_m)\to\infty$.
\end{theorem}

\begin{proof}
The upper bound~\eqref{eq:upper_bound} gives $\nu_k(M_m)\le\lambda_k(L_m)(1+O(m^{-1}))$ for all $k\le N-1$. The lower bound~\eqref{eq:lower_bound} applies equally to all $k\le N-1$: the min-max/dimension argument in part~(E) is index-independent, and the parasitic bound in Case~2 gives $R(u)\ge c_0/m^2\gg\lambda_k(L_m)$ for any fixed~$k$. Combining with Lemma~\ref{lem:scaling} yields the stated asymptotics for all $1\le k\le N-1$.

The bound $\nu_N(M_m)\ge c_0/m^2$ is Proposition~\ref{prop:parasitic}.
\end{proof}

\begin{remark}
The parasitic bound $\nu_N\ge c_0/m^2$ implies
$m^4\nu_N\ge c_0 m^2\to\infty$, whereas the target eigenvalues satisfy
$m^4\nu_k\to\lambda_k^*$ for $k\le n$.
Since $N\ge n+1$, the eigenvalues $\nu_0,\dots,\nu_n$ are eventually separated
from $\nu_N,\nu_{N+1},\dots$ by a gap that diverges after rescaling.
This ensures that the approximation in Theorem~\ref{thm:main} controls
exactly the first $n+1$ eigenvalues, with no unwanted eigenvalues intruding.

The extra discrete eigenvalues $\mu_{n+1},\dots,\mu_{N-1}$ (chosen strictly above
$(\lambda_n^*+1) V_F$ in \S4.3) correspond to continuous eigenvalues of $M_m$
in the intermediate range above the target window and below the parasitic gap.
Their precise location does not affect the first $n+1$ eigenvalues.
\end{remark}

\section{Proof of the Main Theorem}
\label{sec:proof}

\begin{proof}[Proof of Theorem \ref{thm:main}]
Let $0=\lambda^*_0<\lambda^*_1<\cdots<\lambda^*_n$ and $\varepsilon>0$. Without loss of generality, assume $\varepsilon<1$ (it suffices to prove the theorem for $\varepsilon':=\min(\varepsilon,1)$, since an $\varepsilon'$-approximation is automatically an $\varepsilon$-approximation).

\smallskip
\noindent\textbf{Case $d=2$.}
Choose $N=\max(n+1,4)$. If $N>n+1$, extend the target list to length $N$ by appending strictly increasing values $\lambda^*_{n+1}<\cdots<\lambda^*_{N-1}$ with $\lambda^*_{n+1}>\lambda^*_n+1$; this ensures the padded eigenvalues lie well above the target window $[0,\lambda_n^*+\varepsilon]$ for any $\varepsilon<1$.
Apply Lemma~\ref{lem:cdv_complete} (with Remark~\ref{rem:constant_measure}) to $K_N$ with vertex areas $V_v=2\pi(N-3)$ to obtain weights $w_e>0$ realizing the target discrete spectrum $(\lambda^*_1,\dots,\lambda^*_{N-1})$. (Concretely: the $\nu\equiv 1$ lemma is applied to the list $(V_v\lambda^*_1,\dots,V_v\lambda^*_{N-1})$; the resulting Laplacian with constant measure $V_v$ then has eigenvalues $\lambda^*_k$ by the scaling in Remark~\ref{rem:constant_measure}.)
Construct the Fenchel-Nielsen pinching family $\Sigma_\delta$ as in Section~\ref{sec:surfaces}, with collar lengths $\ell_e(\delta)=\pi\delta w_e$ and all remaining Fenchel-Nielsen coordinates fixed.
By Theorem~\ref{thm:burger}, $\lambda_k(\Sigma_\delta)=\delta\lambda_k^*(1+o(1))$ for $1\le k\le N-1$ and all higher eigenvalues are $\ge C_0$.
Rescale the metric by $g_\delta:=\delta g_{\Sigma_\delta}$. Then $\kappa_\delta=-1/\delta\to -\infty$ and $\lambda_k(\Sigma_\delta,g_\delta)=\lambda_k^*+o(1)$ for $k\le n$. For the padded eigenvalues $n+1\le k\le N-1$, the same rescaling gives $\lambda_k(\Sigma_\delta,g_\delta)\to\lambda_k^*>\lambda_n^*+1$, so these also lie above the target window. For $k\ge N$, the rescaled eigenvalues satisfy $\lambda_k(\Sigma_\delta,g_\delta)=\delta^{-1}\lambda_k(\Sigma_\delta)\ge C_0/\delta\to\infty$, so no parasitic eigenvalues intrude. Choosing $\delta$ small yields the $\varepsilon$-approximation.

\smallskip
\noindent\textbf{Case $d\ge 3$.}
Fix $N\ge \max(n+1,5)$ odd and $D=(N-1)/2$ and build the base manifold $B$ and block $F$ as in Section~4.1. If $N>n+1$, the discrete spectrum of $K_N$ has $N-1$ non-zero eigenvalues; we prescribe the first $n$ to match $\lambda_k^* V_F$ and choose the remaining $N-1-n$ to be a strictly increasing sequence strictly above $(\lambda_n^*+1) V_F$, ensuring these ``padding'' eigenvalues lie well above the target window $[0,\lambda_n^*+\varepsilon]$ for any $\varepsilon<1$. (In particular, the discrete weights $w_e^*$ are independent of~$\varepsilon$ and remain fixed during the limit $m\to\infty$.)
Use Lemma~\ref{lem:cdv_complete} to choose weights $w_e^*$ on $K_N$ realizing the full discrete spectrum $(\mu_1,\dots,\mu_{N-1})$, where $\mu_k=\lambda_k^*V_F$ for $k\le n$ and the remaining $\mu_k$ are the padding values.
For large $m$, build the hyperbolic cover $M_m\to B$ (curvature $-1$) with cluster size $m^3$ and corridor lengths $K_e(m)$.
By Theorem~\ref{thm:reduction} and Lemma~\ref{lem:scaling}, $m^4\nu_k(M_m)\to \lambda_k^*$ for each fixed $k\le n$, and $m^4\nu_N(M_m)\to\infty$ (by Theorem~\ref{thm:reduction}), while $m^4\nu_k(M_m)\to\mu_k/V_F>\lambda_n^*+1$ for $n+1\le k\le N-1$.
Now rescale the metric by $g_m:=m^{-4}g_{M_m}$, so $\kappa_m=-m^4\to -\infty$ and $\lambda_k(M_m,g_m)=m^4\nu_k(M_m)\to\lambda_k^*$. Choose $m$ large so that $\lvert \lambda_k(M_m,g_m)-\lambda_k^*\rvert<\varepsilon$ for $k\le n$.

\smallskip
\noindent\textbf{Simplicity.}
For $k<n$: since $\lambda_k^*<\lambda_{k+1}^*$ and $\lambda_k(M_m,g_m)\to\lambda_k^*$ for each $k$, the strict separation $\lambda_{k+1}(M_m,g_m)-\lambda_k(M_m,g_m)>0$ holds for all sufficiently large~$m$. We choose $m$ large enough that this holds simultaneously for all $k<n$ (in addition to the approximation $|\lambda_k-\lambda_k^*|<\varepsilon$ for $k\le n$).
For $k=n$, we must show $\lambda_{n+1}$ is eventually separated from~$\lambda_n$. By the approximation step, $\lambda_n<\lambda_n^*+\varepsilon$. It remains to show $\lambda_{n+1}>\lambda_n^*+\varepsilon$.

If $N>n+1$: in $d=2$, the first padded eigenvalue satisfies
$\lambda_{n+1}(\Sigma_\delta,g_\delta)\to\lambda_{n+1}^*>\lambda_n^*+1$;
in $d\ge 3$, $\lambda_{n+1}(M_m,g_m)=m^4\nu_{n+1}(M_m)\to\mu_{n+1}/V_F>\lambda_n^*+1$ (by Theorem~\ref{thm:reduction}). In either case, $\lambda_{n+1}>\lambda_n^*+\varepsilon$ eventually (since $\varepsilon<1$).

If $N=n+1$: in $d=2$, $\lambda_{n+1}(\Sigma_\delta,g_\delta)\ge C_0/\delta\to\infty$; in $d\ge 3$, $\lambda_{n+1}(M_m,g_m)=m^4\nu_N(M_m)\ge c_0 m^2\to\infty$ (by the parasitic bound, Theorem~\ref{thm:reduction}).

Beyond the padding range (when $N>n+1$): in $d=2$, $\lambda_k(\Sigma_\delta,g_\delta)\ge C_0/\delta\to\infty$ for $k\ge N$; in $d\ge 3$, $m^4\nu_N(M_m)\to\infty$. Hence all $n+1$ eigenvalues in the window $[0,\lambda_n^*+\varepsilon]$ are simple.

\smallskip
\noindent\textbf{Topology.}
In $d=2$ the genus is $\gamma=1+\frac{N(N-3)}{2}$, depending only on $n$.
In $d\ge 3$ the covering degree is $|S|=Nm^3+O(m)$, which grows as $\varepsilon\to 0$ (see Remark~\ref{rem:covering_degree} below).

\smallskip
\noindent\textbf{Obstruction.}
The ``Moreover'' clause of Theorem~\ref{thm:main} (the universal bound $\lambda_1\le\Lambda_d|\kappa|$) is proved in Section~\ref{sec:obstruction}: Proposition~\ref{thm:yang_yau_bound} for $d=2$ and Proposition~\ref{prop:KM_Cheng} for $d\ge 3$.

\end{proof}

\begin{remark}[Covering degree and $\varepsilon$-dependence]
\label{rem:covering_degree}
In the $d\ge 3$ construction, the lower-bound error rate (Remark~\ref{rem:error_rate}) is $O(m^{-1/2})$ with the crude cluster-energy bound, so achieving $\varepsilon$-approximation is guaranteed once $m$ is of order $\varepsilon^{-2}$, giving a covering degree of $O(\varepsilon^{-6})$. (We expect $O(\varepsilon^{-3})$ to be achievable via the tighter cluster-energy analysis sketched in Remark~\ref{rem:error_rate}, but this would require a quantitative energy-partition lemma that we do not prove here.) In particular, the covering degree (and hence the unrescaled volume at curvature $-1$) of the approximating sequence grows polynomially in $\varepsilon^{-1}$; the polynomial depends on the discrete weights $w_e^*$ (determined by the target list through Lemma~\ref{lem:cdv_complete}).
\end{remark}

\section{Geometric Obstructions to Bounded Curvature}
\label{sec:obstruction}

We show that curvature divergence is forced when the target contains values above the universal bounds for $\kappa=-1$ hyperbolic manifolds.

\subsection{Surfaces (\texorpdfstring{$d=2$}{d=2})}

\begin{proposition}
\label{thm:yang_yau_bound}
Let $(M,g)$ be a closed orientable surface of genus $\gamma\ge 2$ with constant sectional curvature $\kappa<0$. Then
\[
\lambda_1(M,g)\le 6|\kappa|.
\]
\end{proposition}

\begin{proof}
By the Yang-Yau inequality \cite{YangYau1980} (Proposition in \S 2), if $M$ admits a conformal branched cover of $S^2$ of degree $d$, then
\[
\lambda_1(M,g)\le \frac{8\pi d}{A},
\]
where $A=\operatorname{Area}(M,g)$.
By the Riemann-Roch theorem, every compact Riemann surface of genus $\gamma$ admits a non-constant meromorphic function of degree at most $\gamma+1$: for any point $p\in M$, $h^0((\gamma+1)p)\ge (\gamma+1)-\gamma+1=2$, giving a non-constant element of $\mathcal{L}((\gamma+1)p)$.
By Gauss-Bonnet, $A=4\pi(\gamma-1)/|\kappa|$, hence
\[
\lambda_1(M,g)
\le
\frac{8\pi(\gamma+1)}{4\pi(\gamma-1)/|\kappa|}
=
2|\kappa|\frac{\gamma+1}{\gamma-1}.
\]
Since $(\gamma+1)/(\gamma-1)\le 3$ for all $\gamma\ge 2$ (with equality at $\gamma=2$), we obtain $\lambda_1\le 6|\kappa|$.
\end{proof}

\subsection{Higher dimensions (\texorpdfstring{$d\ge 3$}{d>=3})}

\begin{proposition}
\label{prop:KM_Cheng}
For each $d\ge 3$ there exists $\Lambda_d>0$ such that any closed $d$-manifold $(M,g)$ of constant sectional curvature $\kappa<0$ satisfies
\[
\lambda_1(M,g)\le \Lambda_d |\kappa|.
\]
\end{proposition}

\begin{proof}
Rescale $g=|\kappa|^{-1}\tilde g$ so that $(M,\tilde g)$ has curvature $-1$. By a corollary of the Kazhdan-Margulis theorem \cite{KazhdanMargulis}, $\Vol(M,\tilde g)\ge v_d>0$.
Since $M$ is contained in a geodesic ball of radius $\diam(M,\tilde g)$, the Bishop-Gromov comparison gives $v_d\le\Vol(M,\tilde g)\le\Vol(B_{-1}(\diam(M,\tilde g)))$, where $B_{-1}(r)$ denotes a ball of radius $r$ in $\mathbb{H}^d$. Since $r\mapsto\Vol(B_{-1}(r))$ is a strictly increasing function with $\Vol(B_{-1}(r))\to 0$ as $r\to 0$, there exists a unique $R_d>0$ satisfying $\Vol(B_{-1}(R_d))=v_d$, and monotonicity forces $\diam(M,\tilde g)\ge R_d$. Hence $\diam(M,g)\ge R_d|\kappa|^{-1/2}$.

Choose two points $p,q\in M$ with $d_{\tilde g}(p,q)=\diam(M,\tilde g)$. Since
$\diam(M,\tilde g)\ge R_d$, the open balls
\[
U_1:=B_{\tilde g}(p,R_d/2), \qquad U_2:=B_{\tilde g}(q,R_d/2)
\]
are disjoint. For $x_1:=p$ and $x_2:=q$, define Lipschitz cutoffs
\[
\eta_j(x):=
\begin{cases}
1, & d_{\tilde g}(x,x_j)\le R_d/4,\\[2mm]
2-\dfrac{4}{R_d}d_{\tilde g}(x,x_j), & R_d/4<d_{\tilde g}(x,x_j)<R_d/2,\\[2mm]
0, & d_{\tilde g}(x,x_j)\ge R_d/2.
\end{cases}
\]
Then $\eta_j\in H^1(M)$, $\supp \eta_j\subset U_j$, and
$|\nabla \eta_j|\le 4/R_d$ almost everywhere. Hence
\[
\int_M |\nabla \eta_j|^2 \le \frac{16}{R_d^2}\Vol(B_{\tilde g}(x_j,R_d/2)),
\qquad
\int_M \eta_j^2 \ge \Vol(B_{\tilde g}(x_j,R_d/4)).
\]
By Bishop--Gromov,
\[
\frac{\Vol(B_{\tilde g}(x_j,R_d/2))}
     {\Vol(B_{\tilde g}(x_j,R_d/4))}
\le
\frac{\Vol(B_{-1}(R_d/2))}
     {\Vol(B_{-1}(R_d/4))}.
\]
Therefore each $\eta_j$ has Rayleigh quotient bounded by
\[
C_d:=
\frac{16}{R_d^2}
\frac{\Vol(B_{-1}(R_d/2))}
     {\Vol(B_{-1}(R_d/4))}.
\]

Let $a_j:=\int_M \eta_j$, and set
\[
\psi:=a_2\eta_1-a_1\eta_2.
\]
Since the supports are disjoint, $\int_M \psi=0$ and
\[
\frac{\int_M |\nabla \psi|^2}{\int_M \psi^2}
=
\frac{a_2^2\int_M |\nabla \eta_1|^2 + a_1^2\int_M |\nabla \eta_2|^2}
     {a_2^2\int_M \eta_1^2 + a_1^2\int_M \eta_2^2}
\le C_d.
\]
By the min-max characterization, $\lambda_1(M,\tilde g)\le C_d$. Scaling back,
\[
\lambda_1(M,g)=|\kappa|\,\lambda_1(M,\tilde g)\le C_d\,|\kappa|.
\]
Thus the claim holds with $\Lambda_d:=C_d$.
\end{proof}

Together, Proposition~\ref{thm:yang_yau_bound} and Proposition~\ref{prop:KM_Cheng} show that if one insists on $\kappa\equiv -1$, then $\lambda_1$ is universally bounded. Hence any target list with $\lambda_1^*$ above this bound cannot be approximated to arbitrary precision in the normalized curvature class $\kappa\equiv-1$. More generally, any approximating sequence for a fixed target must satisfy $\liminf|\kappa|\ge \lambda_1^*/\Lambda_d$; in particular, accommodating arbitrarily large prescribed $\lambda_1^*$ forces $|\kappa|\to\infty$.

\section{Arbitrarily Precise Prescription of Eigenvalue Ratios}
\label{sec:ratios}

While absolute large eigenvalues are obstructed for $\kappa\equiv -1$, scale-invariant ratios can be prescribed.

\begin{corollary}[Arbitrarily Precise Prescription of Eigenvalue Ratios]
\label{cor:eigenvalue_ratios}
For any $d\ge 2$, any strictly increasing sequence $1=\mu_1^*<\cdots<\mu_n^*$, and any $\varepsilon>0$, there exists a closed $d$-manifold $(M,g)$ with $\kappa\equiv -1$ such that
\[
\left|\frac{\lambda_i(M,g)}{\lambda_1(M,g)}-\mu_i^*\right|<\varepsilon
\qquad (i=1,\dots,n).
\]
\end{corollary}

\begin{proof}
Set $\delta=\min\left(\frac12,\frac{\varepsilon}{2+2\mu_n^*}\right)$.
Apply Theorem~\ref{thm:main} to the list $\lambda_0^*=0$ and $\lambda_i^*=\mu_i^*$ to obtain $(M,\tilde g)$ with curvature $\kappa<0$ such that $\lambda_i(\tilde g)=\mu_i^*+E_i$ with $|E_i|\le\delta$.
Rescale to $g=|\kappa|\tilde g$, giving curvature $-1$ and $\lambda_i(g)=\lambda_i(\tilde g)/|\kappa|$, so ratios are unchanged:
\[
\frac{\lambda_i(g)}{\lambda_1(g)}=\frac{\mu_i^*+E_i}{1+E_1}.
\]
Then
\[
\left|\frac{\lambda_i(g)}{\lambda_1(g)}-\mu_i^*\right|
=
\left|\frac{E_i-\mu_i^*E_1}{1+E_1}\right|
\le
\frac{|E_i|+\mu_i^*|E_1|}{1-|E_1|}
\le
\frac{\delta(1+\mu_n^*)}{1-\delta}
\le
2\delta(1+\mu_n^*)\le \varepsilon.
\]
\end{proof}

\section{A Concrete Example: Approximating \texorpdfstring{$\{0,1,3\}$}{\{0,1,3\}} in \texorpdfstring{$d=2$}{d=2}}
\label{sec:example}

To illustrate the construction, we use a path graph $P_3: v_1-v_2-v_3$. For this specific two-eigenvalue target, $P_3$ suffices and yields genus $\gamma=3$. (The generic complete-graph construction with $N=\max(n+1,4)=4$ and $K_4$ also gives genus $\gamma=3$; the advantage of $P_3$ is not a genus reduction but the explicit algebraic tractability of the discrete inverse problem on a path graph, which bypasses the full Colin de Verdi\`ere machinery of Lemma~\ref{lem:cdv_complete}. We note that $P_3$ with the vertex volumes below can only realize targets satisfying $\lambda_2^*\ge 2\lambda_1^*$: the trace and product-of-minors equations determine $w_{12},w_{23}$ as roots of a quadratic whose discriminant factors as $\frac{8\pi^2}{9}(2\lambda_1^*-\lambda_2^*)(\lambda_1^*-2\lambda_2^*)$, and since $\lambda_2^*>\lambda_1^*>0$ this is non-negative precisely when $\lambda_2^*\ge 2\lambda_1^*$. For general targets one must use the complete graph $K_N$, which has enough algebraic degrees of freedom to realize \emph{any} strictly increasing list.)

Assign vertex areas:
\[
V_1=V_3=2\pi,\qquad V_2=4\pi,
\]
corresponding to a torus with one hole at $v_1,v_3$ and a torus with two holes at $v_2$.

With edge weights $w_{12},w_{23}>0$, the weighted Laplacian matrix is
\[
 L_G = \begin{pmatrix} \frac{w_{12}}{2\pi} & -\frac{w_{12}}{2\pi} & 0 \\[4pt]
 -\frac{w_{12}}{4\pi} & \frac{w_{12}+w_{23}}{4\pi} & -\frac{w_{23}}{4\pi} \\[4pt]
 0 & -\frac{w_{23}}{2\pi} & \frac{w_{23}}{2\pi} \end{pmatrix}.
\]
Requiring non-zero eigenvalues $1$ and $3$ gives:
\[
\operatorname{tr}(L_G)=\frac{3}{4\pi}(w_{12}+w_{23})=4,
\qquad
\sum \text{principal }2\times2\text{ minors of } L_G = \frac{1}{2\pi^2}w_{12}w_{23}=3.
\]
Solving yields
\[
 w_{12} = \frac{\pi(8 + \sqrt{10})}{3},
 \qquad
 w_{23} = \frac{\pi(8 - \sqrt{10})}{3},
\]
which are strictly positive (numerically, $w_{12}\approx 11.7$, $w_{23}\approx 5.1$).

Glue the three vertex surfaces accordingly. Each collar is an annulus ($\chi=0$), so the total Euler characteristic is $\chi=\chi(X_1)+\chi(X_2)+\chi(X_3)=-1-2-1=-4$, hence genus $\gamma=3$.
Fix a hyperbolic metric on this genus-$3$ surface and pinch the two separating geodesics to lengths $\ell_1(\delta)=\pi\delta w_{12}$ and $\ell_2(\delta)=\pi\delta w_{23}$, holding all remaining Fenchel-Nielsen coordinates fixed.
By Theorem~\ref{thm:burger} (applied to the path graph $P_3$),
\[
\lambda_0(\Sigma_\delta)=0,\qquad \lambda_1(\Sigma_\delta)=\delta+o(\delta),\qquad \lambda_2(\Sigma_\delta)=3\delta+o(\delta).
\]
Rescaling the metric by $\delta$ produces constant-curvature metrics with $\kappa\to -\infty$ and eigenvalues converging to $\{0,1,3\}$.

\subsection*{Acknowledgements} The author is deeply grateful to Anton Petrunin for general advice, particularly for suggesting that the author look at the work of Vedrin \v{S}ahovic \cite{Sahovic}. The author wishes to thank IIT Bombay for providing ideal working conditions.

\end{document}